\documentclass[11pt]{amsart}
\usepackage{graphicx, xcolor}
\usepackage{amssymb}
\numberwithin{equation}{section}

\def\N{\mathbb{N}}
\def\R{\mathbb{R}}

\newcommand{\rdbrack}{]\!]}
\newcommand{\ldbrack}{[\![}

\newcommand{\bpsi}{{\psi}}

\def\l{\lambda}

\def\epsilon{\varepsilon}
\def\e{\varepsilon}

\newcommand\br{\begin{rem}}
\newcommand\er{\end{rem}}
\newcommand\bp{\begin{pmatrix}}
\newcommand\ep{\end{pmatrix}}
\newcommand\be{\begin{equation}}
\newcommand\ee{\end{equation}}
\newcommand\ba{\begin{equation}\begin{aligned}}
\newcommand\ea{\end{aligned}\end{equation}}

\newtheorem{theorem}{Theorem}[section]
\newtheorem{proposition}[theorem]{Proposition}
\newtheorem{lemma}[theorem]{Lemma}

\newtheorem{example}[theorem]{Example}
\newtheorem{remark}[theorem]{Remark}
\newtheorem{ans}[theorem]{Definition}

\title{Asymptotic behavior of interface solutions to quasilinear parabolic equations with nonlinear forcing terms} 

\begin{document}

\begin{center}
LINDA MARIA DE CAVE\footnote{Sapienza Universit\`a di Roma, Dipartimento S.b.a.i., Roma (Italy). E-mail addresses: \texttt{linda.decave@sbai.uniroma1.it}, \texttt{lindamdecave@gmail.com}}, MARTA STRANI\footnote{Universit\'e Paris-Diderot, Institut de Math\'ematiques de Jussieu-Paris Rive Gauche, Paris (France). E-mail addresses: \texttt{marta.strani@imj-prg.fr}, \texttt{martastrani@gmail.com}.}
\end{center}

\vskip1cm

\begin{abstract}
We investigate the asymptotic behavior of solutions for  quasilinear parabolic equations in  bounded intervals. In particular, we are concerned with a special class  of solutions, called interface solutions, which exhibit e metastable behavior, meaning that their convergence towards the asymptotic configuration of the system is exponentially slow. The key of our analysis is a linearization around an approximation of the steady state of the problem, and the reduction of the dynamics to a one-dimensional motion, describing the slow convergence of the interfaces towards the equilibrium.
\end{abstract}

\maketitle



\pagestyle{myheadings}
\thispagestyle{plain}

\section{Introduction}

In this paper we study the asymptotic behavior  of interface solutions to the  initial-boundary value problem for quasilinear parabolic equations, i.e.
\begin{equation}\label{burgers}
 \left\{\begin{aligned}
& \partial_t u = \varepsilon \partial_x\left(a(x) \partial_x u \right)-G(u, \partial_x u),  &\qquad &x \in I=(-\ell,\ell), \ t \geq 0, \\
& a \, u(\pm \ell,t)+ b \,\partial_x u(\pm\ell,t)=u_{\pm}, &\qquad &t \geq 0,\\
& u(x,0)=u_0(x), &\qquad &x \in I,
  \end{aligned}\right.
\end{equation}
for some $\varepsilon, \ell>0$, $u_{\pm} \in \R$ and $a,b \geq0$. Concerning the function $a(x)$, we require $a \in W^{1,\infty}(I)$; as a consequence, there exist positive constants $\alpha, \beta \in \R^+$ such that
\begin{equation}\label{ipoa}
\alpha \leq a(x) \leq \beta, \quad  {\rm for \ all \ } x \in I \ ,
\end{equation}
so that the classical ellipticity and growth conditions are satisfied. Finally, concerning the initial datum $u_0$ and the nonlinear forcing term $G$, we assume  
\begin{equation*}u_0\in L^2(I),\quad G=G(z,w)\in C^1(\R^2).\end{equation*} 
In particular, we focus our attention to the   phenomenon known as {\bf metastability}, whereby the time dependent solution develops into a layered function in a relatively short time (usually of order one), and then converges towards its stable configuration in a time scale that can be extremely long, depending on the size of the viscosity parameter $\e$.

Such behavior has been extensively studied for a large class of one-dimensional evolutive PDEs; to name some of these results, we recall here the area of viscous scalar conservation laws, with the contributions \cite{BerKamSiv01, KimTzav01, LafoOMal95, MS, ReynWard95, Str13}, or phase transition problems, described by the Allen-Cahn and the Cahn-Hilliard equations equation (\cite{AlikBateFusc91, CarrPego89, FuscHale89, OttoRezn06, Pego89, Str13}). 

Results on metastability for systems of scalar equations are less common; the slow motion for  systems of conservation laws has been examined in \cite{KreiKreiLore08}, while in \cite{BetOrlSme11, Ris08} the authors describe the phenomenon of metastability for systems with a gradient structure, with an analysis that is entirely based on energy methods. Finally, we quote \cite{Str12}, where the one dimensional Jin-Xin systems is analyzed.

The bibliography is so rich that it would be impossible to mention everyone.  
\vskip0.2cm
Roughly speaking, the phenomenon of metastability can be summarized as follows: starting from an initial datum $u_0(x)$ which contains $N$ zeroes inside the interval $I$, a layered solution with exactly $N$ layers is formed in an $\mathcal O(1)$ time scale; once this pattern is formed, it starts to move towards its asymptotic stable configuration, but this motion can be extremely slow as the viscosity parameter $\e$ goes to zero. As a consequence, we can distinguish two different time phases in the dynamics; a first transient phase of order one where the internal shock layers are formed, and a subsequent  long time phase where the layers interact until the solution stabilizes to the stable steady state of the system.
\vskip0.2cm
Usually, such behavior is related to the presence of a first small (with respect to $\e$) eigenvalue of the linearized operator around  the steady state (see, for example, \cite{KreiKrei86}). As it is well known, if $\lambda_1^\e$ is {\bf negative}, the steady is stable; if, in addition, $\lambda^\e_1 \to 0$ as $\e \to 0$, the steady state is metastable in the sense that the time dependent solution converges towards it in a time scale that goes to $\infty$ as $\e$ goes to zero. On the contrary, if $\lambda_1^\e$ is {\bf positive}, the steady state in unstable and we will see the solution to "run away" towards a stable configuration; again, this motion will be extremely slow as $\e\to0$. 

\vskip0.2cm 
The aim of this paper is to describe the metastable behavior of solutions to the general class of quasilinear parabolic equation described in \eqref{burgers}.

In the limit $\e \to 0$, equation \eqref{burgers} reduces to the first order hyperbolic equation
\begin{equation}\label{epszero}
\partial_t u = -G(u,\partial_x u),
\end{equation}
complemented with initial datum $u_0(x)$ and appropriate boundary conditions.  As it is well known, the set of solutions to \eqref{epszero} is the one  given by the entropy formulation, in the sense of Kruzkov (see \cite{Kru70}); moreover, the boundary conditions has to be interpreted in a nonclassical way in the sense of \cite{BardLeRoNede79}.

In the case $a(x)\equiv1$, for some special choices of the nonlinear forcing term $G$, it is possible to prove the existence of discontinuous stationary solutions for the inviscid problem \eqref{epszero}, corresponding to stationary solutions with internal layers for the associated viscous problem (see, for example, \cite{MascTerr99} in the case of a  reaction-convection equation); as already stressed before, the corresponding time dependent solutions exhibit a metastable behavior.
\vskip0.2cm

There are a large number of works that have investigated such phenomenon for problem \eqref{burgers} with $a(x)\equiv1$;  for instance, in \cite{MS, Str13,Str13bis}, the authors describe this behavior for different parabolic equations, throughout the slow motion of the internal interfaces of the solutions. 

Motivated by this, we expect that also in the more general setting  $a(x)$ satisfying  \eqref{ipoa},  all the discontinuities that appear at the hyperbolic level $\e=0$ shall eventually  turn out into smooth internal layers, and that a metastable behavior will be observable in the vanishing viscosity limit.

\vskip0.2cm
In order to characterize the slow dynamics of solutions to \eqref{burgers}, we mean to adapt the theory developed in \cite{MS} for parabolic equations to our more general setting;  this strategy dates back the work of J. Carr and R. L. Pego \cite{CarrPego89} and can be summarized as follows.
 
The principal idea is to construct a family $\{ U^\e(x,\xi_i) \}_{\xi_i \in I}$, $i=1,....., N$, of so called {\it approximate} steady states for the problem, and to linearize the original equation around an element of this family. With {\it approximate} steady state for \eqref{burgers}, we refer to a solution that solves the stationary equation up to an error that is small in a sense that will be specified later. The parameters $\xi_i$ represent the location of the interfaces.

The aim of this construction is to separate the two distinct phases of the dynamics. Firstly, we mean to understand what happens far from the steady state solution, when the interfaces are formed; subsequently, we want to follow the evolution of the layered solution towards the asymptotic limit. In particular, we describe such {\it slow} motion by obtaining an equation for the positions of the interfaces $\xi_i(t)$, $i=1,....., N$. For the sake of simplicity, in this paper we restrict our analysis to the case of a solution with a single shock layer located in $\xi \in I$, the general case being similar (see, for example, \cite{CarrPego89}). 

After the family $\{ U^\e \}$ is given, a study of the eigenvalue problem associated to the linearized operator $\mathcal L^\e$ obtained from the linearization is needed. Indeed, we want to describe the dynamics of solutions  located far from the equilibrium configuration of the system and, by performing a spectral analysis of $\mathcal L^\e$, we are able to show that the speed rate of convergence of the solutions towards the asymptotic configuration is small in $\e$.
\vskip0.2cm
We close this Introduction with an overview of the paper. In Section 2 we present the general strategy we develope in order to describe the long time behavior of solutions belonging to a neighborhood of a family of approximate steady states $\{ U^\e(x;\xi)\}_{\xi \in I}$. By linearizing the original equation \eqref{burgers} around an element of the family, i.e. by looking for a solution $u$ on the form $u=U^\e+v$, and by using an adapted version of the projection method in order to remove the growing components of the perturbation $v$, we obtain a coupled system for the variables $(\xi,v)$, whose analysis is performed in the subsequent section. In particular, in Section 3, we state and prove Theorems \ref{T1}, \ref{T1bis}, \ref{T2} and \ref{T3}, providing different estimates for the perturbation $v$, depending on the choice of the nonlinear term $G(u,\partial_x u)$ and on the sign of the first eigenvalue of the linearized operator obtained from the linearization. Specifically, since we are taking into account also the nonlinear terms in $v$ (arising from the linearization), we will show that both the form of $G$ and the sign of $\lambda_1^\e$ influence the speed rate of convergence to zero of the perturbation. The estimates on $v$ will be used to decouple the nonlinear system for the variables $(\xi;v)$: we end up with a one-dimensional equation of motion for the variable $\xi$, whose analysis is addressed  in Proposition \ref{SMshock}.  In particular, the metastable behavior of the solution is  described through  the convergence of the interface location towards its equilibrium configuration. Again, the speed rate of this motion is influenced by the explicit form of $G$ and by the sign of $\lambda_1^\e$.\\
Precisely, our results can be summarized in the following Theorem.

\begin{theorem}
Let $u(x,t)= U^{\varepsilon}(x;\xi(t))+ v(x,t)$ be the solution of the initial-boundary value problem \eqref{burgers} and  let $\xi^* \in I$ such that $U^\e(x;\xi^*)$ is an exact steady state for \eqref{burgers}.  Then, for $\varepsilon$
  sufficiently small, there exists a time $T^{\varepsilon}$, diverging to $+\infty$ for $\e \to 0$, such that, for $t \leq T^\varepsilon$, the perturbation $v$ is converging to zero with a speed rate depending on $\e$, and  the interface location $\xi(t)$ satisfies the estimate
\begin{equation*}
|\xi(t) - \xi^* | \leq |\xi_0| e^{-\beta^\varepsilon t} \ \ {\rm with}\ \ \beta^\varepsilon >0 \ \ {\rm and }\ \ \beta^\varepsilon \to 0 \ {\rm as} \ \varepsilon \to 0.
\end{equation*} 
\end{theorem}

As a consequence, the interface location is converging towards its  equilibrium configuration $\xi^*$ exponentially in time but, since $\beta^\e$ is small in $\e$, this convergence is extremely slow as $\e\to0$. In particular, the solution $u$ remains close to some non equilibrium configuration  for a time $T^\e$ that can be very long when $\e$ is small, before converging towards the steady state of the system, corresponding to $U^\e(x;\xi^*)$.

This result characterizing the couple $(\xi,v)$ gives a good  qualitative explanation of the transition from the metastable state to the finale stable state. Also, since we are analyzing a complete system for the couple $(\xi, v)$, the theory is more complete with respect to previous papers concerning metastability for parabolic problems (see, for example, \cite{MS,ReynWard95,SunWard99} where only an approximation of the system is taken into account).

Finally, in Section 4, we study, as an example, the case of a quasilinear viscous scalar conservation law: in this  case we give an explicit expression for the approximated family $\{ U^\e\}$, that can be used to provide an  asymptotic expression for the speed of convergence of the interface, showing that it is proportional to $e^{-1/\e}$.  Subsequently we analyze spectral properties of the linear operator arising from the linearization around the approximate steady state $U^\e$, proving that the first eigenvalue is negative and exponentially small in $\e$ (precisely of order $e^{-1/\e}$), while the rest of the spectrum is bounded away from zero. This analysis is needed to give evidence of the validity of the assumptions of Theorems \ref{T1}, \ref{T1bis}, \ref{T2} and \ref{T3}, at least in one concrete situation.

\vskip0.2cm

The main difference with respect to previous papers describing metastability for equations of the form \eqref{burgers}, and in particular with the work \cite{MS}, is that here we are considering a larger class of equations, where the form of the forcing term is not explicitly given and the diffusion is quasi-linear. The study of such class of equations could be a first step to address the problem of metastability for nonlinear-diffusion problems, such as the cases of the p-laplacian or  the fractional laplacian diffusion operator. 

Moreover, as already stressed, in this paper we describe the behavior of the complete system for the couple $(\xi,v)$, where also the nonlinear terms arising from the linearization are taken ion account. Since the forcing term $G$ may even depend on the space derivative of the solution, we need an estimate also for the $H^1$- norm of the perturbation $v$. This gives a more clear overview with respect to \cite{MS}, since the complete system better suites the behavior of the solutions to \eqref{burgers}. 

\section{General framework and linearization}

Let us define the nonlinear differential operator 
$$\mathcal P^\e[u]:=\varepsilon\,\partial_x(a(x) \partial_x u) - G(u,\partial_x u),$$ 
that depends singularly on the parameter $\e$, meaning that $\mathcal P^0[u]$ is of lower order. In particular, the evolutive equation \eqref{burgers} can be rewritten as
\begin{equation}\label{mainproblem}
\partial_t u = \mathcal P^\e[u], \quad u\big|_{t=0}=u_0.
\end{equation}
Let us suppose that there exists at least one solution to $\mathcal P^\e[u]=0$, i.e. there exists a steady state for the problem \eqref{mainproblem}, called here $\bar U^\e(x)$. Our primarily assumption is the following: we suppose that there exists a one-parameter family of functions $\{ U^\varepsilon(x;\xi) \}_{\xi \in I}$ such that
\begin{equation*}
|\langle \psi(\cdot ), \mathcal P^\e[U^\e(\cdot;\xi)]\rangle| \leq \Omega^\e(\xi)|\psi|_{L^\infty}, \quad \forall \, \psi \in C(I), \, \forall \, \xi \in I.
\end{equation*}
being $\Omega^\e(\xi)$ a family of smooth positive functions that converge to zero as $\e \to 0$, uniformly with respect to $\xi$. Moreover, we require that there exists a value $\xi^* \in I$ such that $U^\varepsilon(x;\xi^*)$ is the exact steady state of the problem.

\vskip0.2cm
The family $\{ U^\varepsilon(x;\xi) \}_{\xi \in I}$ can be seen as a {\bf family of approximate steady states} for \eqref{mainproblem}, in the sense that each element satisfies the stationary
equation up to an error that is small in $\e$, and that is measured by $\Omega^\e$. In particular, the parameter $\xi$ describes the unique zero of $U^\e$, corresponding to the location of the interface; if we suppose such parameter to depend  on time, then the evolution of $\xi(t)$ describes the evolution of the solution to \eqref{mainproblem} towards its equilibrium configuration.

\vskip0.2cm

Once the one-parameter family $\{ U^\varepsilon(x;\xi) \}_{\xi \in I}$ is chosen, we look for a solution to \eqref{mainproblem} in the form
\begin{equation}\label{formau}
u(x,t)= U^\e(x;\xi(t))+v(x,t),
\end{equation}
where  the perturbation $v \in C^0(\R^+;H^1(I))$ is determined by the difference between the solution $u$ and an element of the family of approximate steady states. By substituting \eqref{formau} into \eqref{mainproblem}, we obtain
\begin{equation}\label{eqvNL}
\partial_t v = \mathcal L^\e_{\xi(t)} v + \mathcal 	P^\e[U^\e(\cdot;\xi)] -\partial_\xi U^\e(\cdot;\xi) \, \frac{d\xi}{dt}+ \mathcal Q^\e[v,\xi],
\end{equation}
where
\begin{equation*}
	\begin{aligned}
		\mathcal{L}^\varepsilon_\xi v:=d{\mathcal P}^\varepsilon[U^{\varepsilon}(\cdot;\xi)]\,v
	\end{aligned} 
\end{equation*}
is the linearized operator arising from the linearization around $U^\e$, while $\mathcal Q^\e[v,\xi]$ collects  the quadratic  terms in $v$ arising from the linearization and it is defined as
\begin{equation*}
\mathcal{Q}^\varepsilon[v,\xi]:={\mathcal P}^\varepsilon[U^{\varepsilon}(\cdot;\xi)+v]
			-{\mathcal P}^\varepsilon[U^{\varepsilon}(\cdot;\xi)]
				-d{\mathcal P}^\varepsilon[U^{\varepsilon}(\cdot;\xi)]\,v.
\end{equation*}
\begin{example}{\rm
Let us consider the case of a quasilinear scalar conservation law, i.e. problem \eqref{burgers} with $G(u,\partial_x u)=\partial_x f(u)$.  We have 
\begin{equation*}
\begin{aligned}
\mathcal L^{\e,f}_{\xi} v := \e \partial_x \left( a(x) \partial_x v\right)-\partial_x \left(f'(U^\e) \, v \right), \qquad 
\mathcal Q^{\e,f}[v,\xi]:= -\frac{1}{2}   \partial_x \left( f''(U^\e) v^2  \right).
\end{aligned}
\end{equation*}
On the contrary, when considering problem \eqref{burgers} with $G(u,\partial_x u)=g(u)$, since the forcing term $G$ depends only on $u$, we obtain
\begin{equation*}
\begin{aligned}
\mathcal L^{\e,g}_{\xi} v := \e \partial_x \left( a(x) \partial_x v\right)- g'(U^\e) \, v , \qquad 
\mathcal Q^{\e,g}[v,\xi]:= -\frac{1}{2}   g''(U^\e) v^2 .
\end{aligned}
\end{equation*}
In particular one has
\begin{equation*}
\mathcal |Q^{\e,f}[v,\xi]|_{{}_{L^1}} \leq C \,  |v|^2_{{}_{H^1}}, \qquad \mathcal |Q^{\e,g}[v,\xi]|_{{}_{L^1}} \leq \ C\, |v|^2_{{}_{L^2}}.
\end{equation*}
The form of the nonlinear terms in $v$ will play a crucial role in the asymptotic behavior of the solution, as we will see in details later on in the is paper; in particular, it effects the speed rate of convergence of the solutions towards the asymptotic limit.
}
\end{example}

\subsection{Spectral hypotheses and the projection method}
We begin by analyzing the spectrum of the linearized operator $\mathcal L^\e_\xi$; we assume such spectrum  to be composed of a decreasing sequence $\{ \l^\e_k(\xi)\}_{k \in \N}$ of {real} eigenvalues such that
\begin{itemize}
\item  $\l^\e_1(\xi) \to 0$ as $\e \to 0$, uniformly with respect to $\xi$.
\item For all $k\geq 2$, $\l^\e_k(\xi)$ are {\bf negative} and there holds
\begin{equation*}
\lambda^\e_1(\xi)- \lambda_2^\e(\xi) \geq C' \ \ \ \ \ \forall \ \xi \in I, 
\end{equation*}
\end{itemize}
Hence, we assume that there is a spectral gap between the first and the second eigenvalue and we ask for $\l_1^\e$ to be small in $\e$, uniformly with respect to $\xi$. 

\begin{remark}{\rm
We note that there are no requests on the sign of the first eigenvalue $\l^\e_1$. Indeed, the metastable behavior is a consequence only of the smallness, with respect to $\e$, of the absolute value of such first eigenvalue (see, for example, \cite{Str13bis}). 
}
\end{remark}

Denoting by  $\varphi^\varepsilon_k=\varphi^\varepsilon_k(\cdot;\xi)$ the  right eigenfunctions of $\mathcal L^\e_\xi$ and by $\psi^\varepsilon_k=\psi^\varepsilon_k(\cdot;\xi)$ the eigenfunctions of the corresponding  adjoint operator  $\mathcal{L}^{\varepsilon,\ast}_\xi$, we set 
\begin{equation*}
	v_k=v_k(\xi;t):=\langle \psi^\varepsilon_k(\cdot;\xi),v(\cdot,t)\rangle. 
\end{equation*}
We now use an adapted version of the {\rm projection method} in order to obtain an equation of motion for the parameter $\xi$. Since we have supposed the first eigenvalue of the linearized operator to be small in $\e$, in order to remove the singular part of the operator $\mathcal L^\e_\xi$ in the limit $\e \to 0$, we  impose $v_1 \equiv 0$. Hence,  the equation for the parameter $\xi(t)$ is chosen in such a way that the unique growing terms in the perturbation $v$ are canceled out.  In formulas
\begin{equation*}
\frac{d}{dt} \langle \psi^\varepsilon_1(\cdot;\xi(t)), v(\cdot,t) \rangle =0
	\qquad\textrm{and}\qquad
	\langle \psi^\varepsilon_1(\cdot;\xi_0), v_0(\cdot)\rangle=0.
\end{equation*}
Using equation \eqref{eqvNL}, we have
\begin{equation*}
	\langle \psi^\varepsilon_1(\xi,\cdot),\mathcal{L}^\varepsilon_\xi v+ {\mathcal P}^\e[U^{\varepsilon}(\cdot;\xi)]
		- \partial_{\xi}U^{\varepsilon}(\cdot;\xi)\frac{d\xi}{dt}+{\mathcal Q}^\varepsilon[v,\xi] \rangle 
		+ \langle  \partial_{\xi}\psi^\varepsilon_1(\xi,\cdot) \frac{d\xi}{dt},v \rangle =0.
\end{equation*}
Since, for small $\varepsilon$, $\langle \psi^\varepsilon_1, {\mathcal L}^\e_{\xi}v \rangle= \lambda^\e_1\langle \psi^\varepsilon_1, v \rangle=0$, 
we obtain a scalar nonlinear differential equation for the variable $\xi$,  that is
\begin{equation}\label{eqxi0}
		\frac{d\xi}{dt}=\frac{\langle \psi^\varepsilon_1(\cdot;\xi),
			{\mathcal P}^\e[U^{\varepsilon}(\cdot;\xi)]+\mathcal{Q}^\varepsilon[v,\xi] \rangle}{ \langle \psi^\varepsilon_1(\cdot;\xi), \partial_{\xi}U^{\varepsilon}(\cdot;\xi) \rangle - \langle  \partial_{\xi}\psi^\varepsilon_1(\cdot;\xi),v \rangle }.
\end{equation}
We notice that if $U^\varepsilon(\cdot;\xi^*)$ is the exact stationary solution, then
\begin{equation*}
	\mathcal{P}^\varepsilon[U^\varepsilon(\cdot;\xi)]
	=\mathcal{P}^\varepsilon[U^\varepsilon(\cdot;\xi)]-\mathcal{P}^\varepsilon[U^\varepsilon(\cdot;\xi^*)]
	\approx\mathcal{L}_\xi^\varepsilon\partial_\xi U^\varepsilon(\cdot;\xi^*)(\xi-\xi^*).
\end{equation*}
Hence, at least for small $\e$,  the first eigenfunction $\psi^\e_1$ is not transversal to $\partial_\xi U^\varepsilon$ and we can take advantage of the renormalization
\begin{equation*}
\langle \psi^\varepsilon_1(\cdot;\xi), \partial_{\xi}U^{\varepsilon}(\cdot;\xi) \rangle  =1, \quad \forall \ \xi \in I.
\end{equation*}
Since we consider  a small perturbation, in the regime $v \to 0$ we have
\begin{equation*}
	\frac{1}{1-\langle  \partial_{\xi}\psi^\varepsilon_1(\cdot;\xi),v \rangle} 
		= 1+\langle \partial_{\xi} \psi^\varepsilon_1(\cdot;\xi),v \rangle+ R[v],
\end{equation*}
where the remainder $R$ is of order $o(|v|)$, and it is defined as
\begin{equation*}
R[v]:= \frac{\langle  \partial_{\xi}\psi^\varepsilon_1(\cdot;\xi),v \rangle^2}{1-\langle  \partial_{\xi}\psi^\varepsilon_1(\cdot;\xi),v \rangle}.
\end{equation*}
Inserting in \eqref{eqxi0}, we end up with the following nonlinear equation for $\xi$
\begin{equation}\label{eqxiNL}
	\frac{d\xi}{dt}=\theta^\varepsilon(\xi)\bigl(1+\langle\partial_{\xi} \psi^\varepsilon_1, v \rangle\bigr)
		+ \rho^\varepsilon[\xi,v], 
\end{equation}
where
\begin{equation*}
	\begin{aligned}
 	\theta^\varepsilon(\xi)
		&:=\langle \psi^\varepsilon_1,{\mathcal P^\e[U^{\varepsilon}] \rangle},\\
	\rho^\varepsilon[\xi,v]&:=  \langle \psi^\e_1, \mathcal Q^\e[v,\xi] \rangle \bigl(1+\langle\partial_{\xi} \psi^\varepsilon_1, v \rangle\bigr)+ \langle \psi^\e_1, \mathcal P^\e[U^\e]+ \mathcal Q^\e[v,\xi] \rangle R[v].
	\end{aligned}
\end{equation*}
Moreover, plugging \eqref{eqxiNL} into \eqref{eqvNL}, we obtain a partial differential equation for the perturbation $v$
\begin{equation*}
\partial_t v= H^\varepsilon(x;\xi)
		+ ({\mathcal L}^\varepsilon_\xi+{\mathcal M}^\varepsilon_\xi)v
			+\mathcal{R}^\varepsilon[v,\xi],
\end{equation*}
where 
\begin{align*}
		H^\varepsilon(\cdot;\xi)&:={\mathcal P}^\varepsilon[U^{\varepsilon}(\cdot;\xi)]
			-\partial_{\xi}U^{\varepsilon}(\cdot;\xi)\,\theta^\varepsilon(\xi),\\
		{\mathcal M}^\varepsilon_\xi v&:=-\partial_{\xi}U^{\varepsilon}(\cdot;\xi)
			\,\theta^\varepsilon(\xi)\,\langle\partial_{\xi} \psi^\varepsilon_1, v \rangle,\\ 
		\mathcal{R}^\varepsilon[v,\xi]&:=\mathcal{Q}^\varepsilon[v,\xi]
								-\partial_{\xi}U^{\varepsilon}(\cdot;\xi)\,\rho^\varepsilon[\xi,v].
\end{align*}

\section{The metastable dynamics}

The couple $(v,\xi)$ solves the system 
\begin{equation}\label{NL}
 	\left\{\begin{aligned}
	\frac{d\xi}{dt}&=\theta^\varepsilon(\xi)\bigl(1 
		+\langle\partial_{\xi} \psi^\varepsilon_1, v \rangle\bigr) + \rho^\varepsilon[\xi,v] , \\
	\partial_t v &= H^\varepsilon(\xi)+ ({\mathcal L}^\varepsilon_\xi+{\mathcal M}^\varepsilon_\xi)v+\mathcal{R}^\varepsilon[v,\xi],
 	\end{aligned}\right. 
\end{equation}
with initial conditions given by
\begin{equation*}
	\langle \psi^\varepsilon_1(\cdot;\xi_0), u_0-U(\cdot;\xi_0)\rangle =0, \qquad v_0=u_0-U(\cdot;\xi_0).
\end{equation*}
Our aim is to describe the behavior of the solution to \eqref{NL} in the regime
of small $\varepsilon$. 

As stated in the introduction, the asymptotic behavior of the solution and, in particular, the speed rate of convergence of the shock layer towards the equilibrium configuration, is strictly related to the specific form of the nonlinear terms arising from the linearization of the original problem around the family $U^\e$. To be more precise, for a certain class of parabolic equations (as, for example,  viscous conservation laws), these quadratic terms involve a dependence on the space derivative of the solution, so that an additional bound for the $L^2$-norm of the space derivative of the solution is needed. On the contrary, when considering equations where the forcing term depends only on the solution itself (as, for instance, equation of reaction-diffusion type), we need to establish an upper bound only for the $L^2$-norm of $v$. 

Furthermore, an important role is played by the first eigenvalue of the linearized operator; indeed, heuristically, the large time  behavior of solutions is described by terms of order $e^{\lambda_1^\varepsilon\,t}$. In particular, the sign of $\lambda^\e_1$ characterizes the stability properties of the steady state around which we are linearizing. When such eigenvalue is {\bf negative}, the steady state is stable and the solution is metastable in the sense that,  starting from an initial configuration located {\bf far} from the equilibrium, the time-dependent solution starts to drifts in an exponentially long time towards the asymptotic limit. On the other side, when $\lambda^\e_1$ is {\bf positive}, the solutions is said to be metastable because, starting from an initial datum located {\bf near} the unstable steady state, the solution drifts apart  towards one of the stable equilibrium configurations of the system, and this motion is extremely slow.
\vskip0.3cm
Hence, we need to distinguish different situations, depending on the type of equation we are dealing with; precisely, we will obtain different estimates for the perturbation $v$ and for the speed rate of convergence of the shock layer (dictated by the behavior of $\xi(t)$), solutions to \eqref{NL}, depending on the sign of $\lambda^\e_1$ and on the form of the nonlinear term $\mathcal Q^\e$. Before state our results, let us recall the hypotheses we need.

\vskip0.2cm
{\bf H1.} The family $\{U^{\varepsilon}(\cdot,\xi)\}$ is such that there exist smooth functions $\Omega^\varepsilon(\xi)$ such that
\begin{equation}\label{defomegaeps}
	|\langle \psi(\cdot),{\mathcal P}^\varepsilon[U^{\varepsilon}(\cdot,\xi)]\rangle|
		\leq |\Omega^\varepsilon(\xi)|\,|\psi|_{{}_{\infty}} \qquad \forall\,\psi\in C(I),
\end{equation}
with $\Omega^\varepsilon$ converging to zero as $\varepsilon\to 0$, uniformly with respect to $\xi\in I$. Moreover, we require that there exists a value $ \xi^* \in I$ such that the element $U^\e(x; \xi^*)$ corresponds to an exact steady state for the original equation.
\vskip0.2cm
{\bf H2.} The eigenvalues $\{\lambda^\varepsilon_k(\xi)\}_{{}_{k\in\N}}$ of the linearized operator $\mathcal{L}^\varepsilon_\xi$ are real and such that 
\begin{equation*}
	\lim_{\varepsilon \to 0}\lambda^\varepsilon_1(\xi) =0, \quad \lambda^\varepsilon_1(\xi)- \lambda^\varepsilon_k(\xi)>c_1 \quad  {\rm and } \quad \lambda_k^\e(\xi) \leq  -\frac{c_2}{\e^{\alpha}}
	\quad \textrm{for }k\geq 2.
\end{equation*}
for some constants $c_1, c_2>0$ independent on $k\in\N$, $\varepsilon>0$ and $\xi\in I$, and for some $\alpha \geq 0$.
\vskip0.2cm
{\bf H3.} There exists a constant $C>0$ such that
\begin{equation*}
|\Omega^\varepsilon(\xi)| \leq C |\lambda^\varepsilon_1(\xi)|, \quad \forall \ \xi  \in I.
\end{equation*}

\subsection{{The case $\lambda^\e_1<0$ and the quadratic term $\mathcal Q^\e$ depending only on $v$}}

At first we consider the case of a nonlinearity $\mathcal Q^\e$ that only depends on the perturbation $v$, and not on its space derivatives; we show that, if we consider a perturbation $v$ such that $v(0,x)$ is bounded, than we can perform an $L^\infty$ estimate for the solution. In order to state an prove our result, we need an additional hypothesis.

\hskip1cm

\vskip0.2cm
{\bf H4.} Concerning the solution $z$ to the linear problem $\partial_t z=\mathcal L^\varepsilon_\xi z$, we require that there exists $\nu^\varepsilon>0$ such that, for all $\xi \in I$, there exist a constant $\bar C$ such that
\begin{equation*}
|z(t)|_{{}_{L^2}} \leq \bar C |z_0|_{{}_{L^2}}e^{-\nu^\varepsilon t}, \quad  \forall \xi \in I
\end{equation*}
\begin{remark}\rm{
The constant $\bar C$ could depend on $\xi$. In this specific case, since $\xi$ belongs to a bounded interval of the real line, if we suppose that $\xi \mapsto C_{\xi(t)}$ is a continuous function, then there exists a maximum of $C_\xi$  in $I$, namely $\bar C$.
}
\end{remark}

\vskip0.5cm

\begin{theorem}\label{T1}
Let hypotheses {\bf H1-2-3-4} be satisfied and let $|v_0|_{{}_{L^\infty}} < +\infty$. Then, for $\varepsilon$ sufficiently small, there exists a time $T^\e$ diverging to $+\infty$ as $\e \to 0$, such that, for all $t \leq T^\e$ the solution $v$ to \eqref{NL} satisfies
\begin{equation*}
|v|_{{}_{L^\infty}}(t) \leq  C|\Omega^\varepsilon|_{{}_{L^\infty}} t+e^{-\mu^\varepsilon t}|v_0|_{{}_{L^\infty}},  \end{equation*}
for some positive constant $C$ and
\begin{equation*}
\mu^\varepsilon :=\sup_{\xi}\{\lambda^\varepsilon_1(\xi)\}- C|\Omega^\varepsilon|_{{}_{L^\infty}}>0.
\end{equation*}
\end{theorem}

The proof of Theorem \ref{T1} we present here is based on the theory of {\it stable families of generators}, firstly developed by Pazy in \cite{Pazy83}; it is a generalization of the theory of semigroups for evolution systems of the form $\partial_t u= L u$, when the linear operator $L$ depends on time. We refer to Appendix A for the definitions of the tools we shall use in the following.

\begin{proof}[Proof of Theorem \ref{T1}]
First of all, let us notice that $\mathcal M^\varepsilon_\xi$ is a bounded operator that satisfies the estimate
\begin{equation}\label{asyMH}
\begin{aligned}
\|\mathcal M^\varepsilon_\xi \|_{\mathcal L(L^2;\R)} \leq c_1|\theta^\varepsilon(\xi)| \leq c_1 |\Omega^\varepsilon|_{{}_{L^\infty}}, \quad \forall \xi \in I.
\end{aligned}
\end{equation}
and $H^\varepsilon(\xi)$ is such that
\begin{equation}\label{asyMH2}
|H^\varepsilon|_{{}_{L^\infty}} \leq c_2| \Omega^\varepsilon|_{{}_{L^\infty}},
\end{equation}
for some positive constants $c_1$ and $c_2$.
Moreover, concerning the nonlinear terms $\rho^\e$ and $\mathcal R^\e$ and because of the specific form of $\mathcal Q^\e$, there follows
\begin{equation}\label{estnonlinearterms}
|\mathcal R^\e|_{{}_{L^\infty}}  \leq C |v|^2_{{}_{L^\infty}}.
\end{equation} 
Next, we want to show that $\mathcal L^\varepsilon_\xi+\mathcal M^\varepsilon_\xi$ is the infinitesimal generator of a $C_0$ semigroup $\mathcal T_\xi(t,s)$. To this aim, concerning the eigenvalues of the linear operator $\mathcal L^\varepsilon_\xi$, we know that $\lambda_1^\varepsilon(\xi)$ is negative and goes to zero as $\varepsilon \to 0$, for all $\xi \in I$. Hence, defining $\Lambda_1^\varepsilon:=\sup_\xi \lambda_1^\varepsilon(\xi)$, we have $\lambda_k^\varepsilon \leq -|\Lambda_1^\varepsilon| <0$ for all $k \geq 1$. Also, for $t \in [0,T]$, $\mathcal L^\varepsilon_{\xi(t)}$ is the infinitesimal generator of a $C_0$ semigroup $\mathcal S_{\xi(t)}(s)$, $s>0$ and,  since {\bf H4}  holds {with the choice $\nu^\e=|\Lambda_1^\varepsilon|$}, we get
\begin{equation*}
\| \mathcal S_{\xi(t)}(s)\| \leq \bar C e^{-|\Lambda_1^\varepsilon|s},
\end{equation*}
and this estimate is independent on $t$. {Thus}, by using Definition \ref{def1} and the following remark, we can state that the family $\{ \mathcal L^\varepsilon_{\xi(t)}\}_{\xi(t) \in J}$ is stable with stability constants $M=\bar C$ and $\omega=-|\Lambda_1^\varepsilon|$. Furthermore, since \eqref{asyMH} holds, Theorem \ref{thpaz1} states that the family $\{ \mathcal L^\varepsilon_{\xi(t)}+ \mathcal M^\varepsilon_{\xi(t)}\}_{\xi(t) \in J}$ is stable with $M=\bar C$ and $\omega= -|\Lambda_1^\varepsilon|+{\bar C}|\Omega^\varepsilon|_{{}_{L^\infty}} $. In particular, by choosing $\bar C=1/C$ , $\omega$ is negative since {\bf H3} holds.

Going further, in order to apply Theorem \ref{thpaz3}, we need to check that the domain of $\mathcal L^\varepsilon_\xi+\mathcal M^\varepsilon_\xi$ does not depend on time;  this is true since $\mathcal L^\varepsilon_\xi+\mathcal M^\varepsilon_\xi$ depends on time through the function $U^\varepsilon(x;\xi(t))$, that does not appear in the higher order terms of the operator. More precisely, the principal part of the operator does not depend on $\xi(t)$.

Hence, we can define $\mathcal T_\xi(t,s)$ as the {\it evolution system} of $\partial_t v=(\mathcal L^\varepsilon_\xi+ \mathcal M^\varepsilon_\xi)v$,
so that
\begin{equation}\label{Yfamiglieevol}
v(t)=\mathcal T_\xi(t,s)v_0+\int_s^t \mathcal T_\xi(t,r)H^\varepsilon(x;\xi(r)) dr+\int_s^t \mathcal T_\xi(t,r)\, \mathcal R^\varepsilon[\xi(r);v(r)] dr, \quad 0 \leq s \leq t 
\end{equation} 
Moreover, there holds
\begin{equation*}
\|\mathcal T_\xi(t,s)\| \leq \bar Ce^{-\mu^\varepsilon (t-s)}, \qquad \mu^\varepsilon := |\Lambda^\varepsilon_1|- {\bar C}|\Omega^\varepsilon|_{{}_{L^\infty}}>0.
\end{equation*}
Finally, from the representation formula \eqref{Yfamiglieevol} with $s=0$ and since \eqref{estnonlinearterms} holds, there follows 
\begin{equation*}
|v|_{{}_{L^\infty}}(t) \leq e^{-\mu^\varepsilon t}|v_0|_{{}_{L^\infty}}+ \sup_{\xi \in I}|H^\varepsilon|_{{}_{L^\infty}}(\xi) \int_0^te^{-\mu^\varepsilon(t-r)} \ dr+ |v|^2 _{{}_{L^\infty}} (t)\int_0^te^{-\mu^\varepsilon(t-r)} \ dr,\quad \forall \, t \geq 0,
\end{equation*}
so that, by using \eqref{asyMH2}, we end up with
\begin{equation*}
|v|_{{}_{L^\infty}}(t) \leq  e^{-\mu^\varepsilon t}|v_0|_{L^\infty}+c_1|\Omega^\varepsilon|_{{}_{L^\infty}}t + c_2 |v|^2 _{{}_{L^\infty}} \, t.
\end{equation*}
Hence, setting $N(t):=|v|_{{}_{L^\infty}}(t) $, we can rewrite the previous inequality as
\begin{equation*}
N(t) \leq A \, N^2(t) + B,
\end{equation*}
and we can conclude $N(t) \leq B$ providing  $4AB \leq 1$. This condition is a condition on the final time $t$ that reads
\begin{equation}\label{tempofinale}
{C_1}t \,  e^{-\mu^\e t} |v_0|^2_{{}_{L^\infty}} + {C_2} |\Omega^\e| \, t^2 \leq 1.
\end{equation}
Precisely, the function $g(t):= {C_1}t \,  e^{-\mu^\e t} |v_0|^2_{{}_{L^\infty}} + {C_2} |\Omega^\e| \, t^2$ behaves like $ |\Omega^\e| \, t^2$ for large $t$; since $|\Omega^\e| \to 0$ as $\e \to 0$, condition \eqref{tempofinale} is satisfied for all $t \leq T^\e$, where $T^\e \to \infty$ as $\e \to 0$.

Under this condition, the final estimate for $v$ reads
\begin{equation}\label{stimafinalev}
|v|_{{}_{L^\infty}}(t) \leq e^{-\mu^\varepsilon t}|v_0|_{L^\infty}+c_1|\Omega^\varepsilon|_{{}_{L^\infty}}t, \quad {\rm for \ all \ } t \leq T^\e
\end{equation}
and the proof is completed.
\end{proof}

Let us stress that, in the proof of Theorem \ref{T1}, the negativity of $\lambda_1^\e$ is crucial in the construction of a {\it stable} family of generators. We also take advantage of  the expression of $\mathcal Q^\e$, where the first order {space} derivative of $v$ does not appear;  indeed, this allows us to estimate the nonlinear term $\mathcal R^\e$ via the $L^\infty$- norm of $v$.

\vskip0.2cm
If we start from an initial datum $v_0$ with a weaker regularity, precisely belonging to $L^2$, we can prove an estimate analogous to \eqref{stimafinalev} for the $L^2$ norm of the solution, but the following additional technical hypothesis is needed.

\vskip0.2cm
{\bf H4.1} Given $\xi\in I$, let $\varphi^\varepsilon_k(\cdot;\xi)$ and $\bpsi^\varepsilon_k(\cdot;\xi)$
be a sequence of eigenfunction for the operators $\mathcal{L}^\varepsilon_{\xi}$ and 
$\mathcal{L}^{\varepsilon,\ast}_{\xi}$ respectively; we assume 
\begin{equation}\label{derpsiphi}
	\sum_{j} \langle \partial_\xi \bpsi^\varepsilon_k, \varphi^\varepsilon_j\rangle^2
	=\sum_{j} \langle \bpsi^\varepsilon_k, \partial_\xi \varphi^\varepsilon_j\rangle^2
	\leq C.
\end{equation}
for all $k$ and for some constant $C$ independent on the parameter $\xi$.

\vskip.15cm
\begin{theorem}\label{T1bis}
Let  the couple $(\xi,v)$ be the solution to initial-value problem \eqref{NL}. If the hypotheses {\bf H1-2-3-4.1} are satisfied, then, for every $\varepsilon$ sufficiently small {there exists a time $T^\e$ such that}, for every ${v}_0 \in L^2(I)$  and for every $t \leq T^\varepsilon$, there holds for the solution ${v}$
\begin{equation*}
|{v}-{z}|_{{}_{L^2}}(t) \leq C \left( |\Omega^\varepsilon|_{{}_{L^\infty}} +\exp\left(\int_0^t \lambda_1^\varepsilon(\xi(\tau)) \, d\tau\right) |{v}_0|^2_{{}_{L^2}}\right),
\end{equation*}
where the function ${z}$ is defined as
\begin{equation*}
{z}(x,t):=\sum_{k\geq 2} v_k(0)\,\exp \left( \lambda_k^\varepsilon(\xi(\tau)) \, d\tau\right)\,\varphi^\varepsilon_k(x;\xi(t)).
\end{equation*}
Moreover, the time $T^\varepsilon$ is of order $  |\sup\limits_{\xi\in I} \lambda_1^\varepsilon (\xi)|^{-1}$, hence diverging to $+\infty$ as $\varepsilon \to 0$.
\end{theorem}

\begin{proof}[Proof of Theorem \ref{T1bis}]
Setting
\begin{equation*}
	{v}(x,t)=\sum_{j} {v}_j(t)\,\varphi^\varepsilon_j(x,{\xi}(t)),
\end{equation*}
we obtain an infinite-dimensional differential system for the coefficients $v_j$
\begin{equation}\label{eqwk_bis}
	\frac{dv_k}{dt}=\lambda^\varepsilon_k(\xi)\,v_k
		+\langle \bpsi^\varepsilon_k,{F_1}\rangle + \langle \bpsi^\varepsilon_k,{F_2}\rangle
\end{equation}
where, omitting the dependencies for shortness,
\begin{equation*}
\begin{aligned}
	{F_1}&:={H}^\varepsilon+\sum_{j} v_j\,\Bigl\{{\mathcal M}^\varepsilon_\xi\, \varphi^\varepsilon_j
			-\partial_\xi \varphi^\varepsilon_j\,\frac{d\xi}{dt}\Bigr\}
		={H}^\varepsilon-\theta^\varepsilon \sum_{j}\Bigl({a}_j+\sum_{\ell} {b}_{j\ell}\,v_\ell\Bigr)v_j, \\
	{F_2}&:= {\mathcal Q}^\varepsilon- \left( \sum_j \partial_\xi \varphi_j^\varepsilon v_j + \partial_\xi {U}^\varepsilon\right) \left \{  \frac{\langle \bpsi_1^\varepsilon,{ \mathcal Q}^\varepsilon \rangle}{1-\langle \partial_\xi \bpsi_1^\varepsilon, {v} \rangle }- \theta^\varepsilon \frac{\langle \partial_\xi \bpsi_1^\varepsilon, {v} \rangle^2}{1-\langle \partial_\xi \bpsi_1^\varepsilon, {v} \rangle }  \right\}= {\mathcal Q}^\varepsilon- {\mathcal N}^\varepsilon.
\end{aligned}
\end{equation*}
The coefficients ${a}_j$, ${b}_{jk}$ are given by
\begin{equation*}
	{a}_j:=\langle \partial_{\xi} \bpsi^\varepsilon_1, \varphi^\varepsilon_j\rangle\,
			\partial_{\xi}{U}^{\varepsilon}+\partial_\xi \varphi^\varepsilon_j,
	\qquad
	{b}_{j\ell}:=\langle \partial_{\xi} \bpsi^\varepsilon_1, \varphi^\varepsilon_\ell\rangle
			\,\partial_\xi \varphi^\varepsilon_j.
\end{equation*}
Convergence of the series is guaranteed by assumption \eqref{derpsiphi}. Now let us set
\begin{equation*}
	E_k(s,t):=\exp\left( \int_s^t \lambda_k^\varepsilon(\xi(\tau))d\tau\right).
\end{equation*}
Note that, for $0\leq s<t$, there hold
\begin{equation*}
	E_k(s,t)=\frac{E_k(0,t)}{E_k(0,s)}
		\qquad\textrm{and}\qquad	
	0\leq E_k(s,t)\leq e^{\Lambda^\e_k(t-s)},\quad\text{{where}}\;{\Lambda_k^\varepsilon := \sup\limits_{\xi\in I} \lambda_k^\varepsilon (\xi)}.
\end{equation*}
From equalities {\eqref{eqwk_bis}} and and since there holds $v_1 \equiv 0$,  there follows
\begin{equation*}
	\begin{aligned}
	v_k(t)&=v_k(0)\,E_k(0,t)\\
		& \quad+\int_0^t \Bigl\{\langle \bpsi^\varepsilon_k,{H}^\varepsilon\rangle
			-\theta^\varepsilon(\xi)\sum_{j}\Bigl(\langle \bpsi^\varepsilon_k, {a}_j\rangle 
			+\sum_{\ell} \langle \bpsi^\varepsilon_k, {b}_{j\ell}\rangle \,v_\ell\Bigr) v_j
			\Bigr\} E_k(s,t)\,ds \\
		& \quad+\int_0^t  \Bigl\{\langle \bpsi^\varepsilon_k,{\mathcal Q}^\varepsilon\rangle-\langle \bpsi^\varepsilon_k,{\mathcal N}^\varepsilon\rangle \bigr\} E_k(s,t) \, ds
	\end{aligned}
\end{equation*}
for $k\geq 2$. Now let us introduce the function 
\begin{equation*}
	{z}(x,t):=\sum_{k\geq 2}v_k(0)\,E_k(0,t)\,\varphi^\varepsilon_k(x;\xi(t)),
\end{equation*}
which satisfies the estimate $|{z}|_{{}_{L^2}}\leq |{v}_0|_{{}_{L^2}} e^{\Lambda^\varepsilon_2\, t}$. From the representation formulas for the coefficients $v_k$, there holds
\begin{equation*}
|{v}-{z}|^2_{{}_{L^2}} \leq \sum_{k \geq 2} \left\{  \int_0^t \Bigl ( |\langle \bpsi_k^\varepsilon, {F} \rangle|+ |\langle \bpsi_k^\varepsilon, {G}\rangle|\Bigr)E_k(s,t) \, ds\right\}^2.
\end{equation*}
Moreover, since
\begin{equation*}
 	|\theta^\varepsilon(\xi)|\leq C\,\Omega^\varepsilon(\xi)
		\quad\textrm{and}\quad
	|\langle \bpsi^\varepsilon_k,{H}^\varepsilon\rangle| \leq C\,\Omega^\varepsilon(\xi)
		\left\{1+ |\langle \bpsi_k^\varepsilon, \partial_\xi {U}^\varepsilon\rangle|\right\}
\end{equation*}
for some constant $C>0$ depending on the $L^\infty-$norm of $\bpsi^\varepsilon_k$,
there holds
\begin{equation*}
	\begin{aligned}
	\sum_{k \geq 2} \Bigl(  \int_0^t &  |\langle \bpsi_k^\varepsilon, {F_1} \rangle|  \, E_k(s,t) \, ds \Bigr)^2 \leq \\
	& \leq C\sum_{k\geq 2}\Bigl(\int_0^t
		 \Omega^\varepsilon(\xi)\Bigl(1+ |\langle \bpsi_k^\varepsilon, \partial_\xi {U}^\varepsilon\rangle|
		+|\langle \bpsi^\varepsilon_k,\partial_{\xi}{U}^{\varepsilon}\rangle|
			\sum_{j}|\langle \partial_{\xi} \bpsi^\varepsilon_1, \varphi^\varepsilon_j\rangle||v_j|\\
	&\quad
		+\sum_{j}|\langle \partial_\xi \bpsi^\varepsilon_k,  \varphi^\varepsilon_j\rangle||v_j|
		+\sum_{j} |\langle \bpsi^\varepsilon_k,\partial_\xi \varphi^\varepsilon_j\rangle|\,|v_j|
		\sum_{\ell}|\langle \partial_{\xi} \bpsi^\varepsilon_1, \varphi^\varepsilon_\ell\rangle|\,|v_\ell|\Bigr)
			E_k(s,t)\Bigr)^2\\
	&
	\leq C\sum_{k\geq 2}\Bigl(\int_0^t
		 \Omega^\varepsilon(\xi)\bigl(1+|{v}|_{{}_{L^2}}^2\bigr)E_k(s,t)\,ds\Bigr)^2.
	\end{aligned}
\end{equation*}
On the other side, concerning the nonlinear terms, there holds
\begin{equation*}
\begin{aligned}
\sum_{k \geq 2} \Bigl(  \int_0^t  & |\langle \bpsi_k^\varepsilon, {F_2} \rangle| \, E_k(s,t) \, ds \Bigr)^2 \leq  \\ & \leq C \sum_{k \geq 2} \Bigl( \int_0^t|\langle\bpsi_k^\varepsilon, {\mathcal Q}^\varepsilon \rangle|  +\left|  \frac{\langle \bpsi_1^\varepsilon,{\mathcal Q}^\varepsilon \rangle}{1-\langle \partial_\xi \bpsi_1^\varepsilon,{v}\rangle}\right|\,\Bigl( \sum_{j} |\langle\partial_\xi \varphi^\varepsilon_j, \bpsi_k^\varepsilon \rangle| |{v}_j|+ |\langle \bpsi_k^\varepsilon, \partial_\xi {U}^\varepsilon\rangle| \Bigl) \\
&\quad+ \Omega^\varepsilon(\xi) \, \left|  \frac{\langle\partial_\xi \bpsi^\varepsilon_1,{v} \rangle^2}{1-\langle \partial_\xi \bpsi_1^\varepsilon,{v}\rangle}\right| \, \Bigl( \sum_{j} |\langle\partial_\xi \varphi^\varepsilon_j, \bpsi_k^\varepsilon \rangle| |{v}_j|+ |\langle \bpsi_k^\varepsilon, \partial_\xi {U}^\varepsilon\rangle|\Bigl) \, ds \, \Bigr)^2.
\end{aligned}
\end{equation*}
Moreover, since $|{\mathcal Q}^\varepsilon|_{{}_{L^1}} \leq C |{v}|^2_{{}_{L^2}}$, we have
\begin{equation*}
      \begin{aligned}
      |\langle\bpsi_k^\varepsilon, {\mathcal Q}^\varepsilon \rangle| &\leq C |{v}|^2_{{}_{L^2}}, \\
       \Bigl|\sum_{j} \langle \bpsi_k^\varepsilon, \partial_\xi \varphi_j^\varepsilon\rangle {v}_j\Bigr|& \leq C |{v}|_{{}_{L^2}},\; \\
      |\langle \bpsi_k^\varepsilon, \partial_\xi {U}^\varepsilon\rangle|\, \frac{|\langle \bpsi_1^\varepsilon,{\mathcal Q}^\varepsilon \rangle|}{|1-\langle \partial_\xi \bpsi_1^\varepsilon,{v}\rangle|} & \leq \frac{C|{v}|^2_{{}_{L^2}}}{1-C|{v}|_{{}_{L^2}} }\leq 2 C |{v}|^2_{{}_{L^2}},\; \\
      |\theta^\varepsilon\langle \partial_\xi \bpsi_1^\varepsilon, {v} \rangle^2 \, \langle \bpsi_k^\varepsilon, \partial_\xi{ U}^\varepsilon\rangle| &\leq C\, \Omega^\varepsilon(\xi) |{v}|^2_{{}_{L^2}},
      \end{aligned}
\end{equation*}
so that we end up with
\begin{equation*}
      \begin{aligned}
      \sum_{k \geq 2} \Bigl(  \int_0^t  |\langle \bpsi_k^\varepsilon, {F_2} \rangle| \, E_k(s,t) \, ds \Bigr)^2 &\leq
 C \sum_{k \geq 2} \Bigl( \int_0^t |{v}|_{{}_{L^2}}^2 (1+ \Omega^\varepsilon(\xi)) \, E_k(s,t) \, ds\Bigr)^2.   
    \end{aligned}
\end{equation*}
Since $\sqrt{a+b}\leq \sqrt{a}+\sqrt{b}$, we infer
\begin{equation*}
	\begin{aligned}
	|{v}-{z}|_{{}_{L^2}}&\leq 
		\sum_{k\geq 2}\int_0^t\left \{\Omega^\varepsilon(\xi)
			\bigl(1+|{v}|_{{}_{L^2}}^2\bigr)+  |{v}|_{{}_{L^2}}^2 \right \}E_k(s,t)\,ds\\
				&\leq 
		C\int_0^t \Bigl\{\Omega^\varepsilon(\xi)\bigl(1+|{v}|_{{}_{L^2}}^2\bigr) + |{v}|^2_{{}_{L^2}}\Bigr\}\,\sum_{k\geq 2} E_k(s,t)\,ds.
	\end{aligned}
\end{equation*}
The assumption on the asymptotic behavior of the eigenvalues $\lambda^\varepsilon_k$
can now be used to bound the series.
Indeed, there holds
\begin{equation*}
	\sum_{k\geq 2} E_k(s,t) \leq E_2(s,t) \sum_{k\geq 2} \frac{E_k(s,t)}{E_2(s,t)}
		\leq C\,(t-s)^{-1/2}\,E_2(s,t).
\end{equation*}
As a consequence,   we infer
\begin{align*}
	 |{v}-{z}|_{{}_{L^2}}
		&\leq  C\int_0^t \Omega^\varepsilon(\xi)(t-s)^{-1/2}\,E_2(s,t) \, ds\\
		&\quad +{C}\int_0^t \Bigl\{ |{v}-{z}|^2_{{}_{L^2}}+|{z}|^2_{{}_{L^2}} \Bigr\}\,(t-s)^{-1/2}\,E_2(s,t)\,ds.
	\end{align*}
Now, setting $N(t):=  \sup\limits_{s\in[0,t]} E_1(s,0)\, |({v}-{z})(s)|_{{}_{L^2}}$, we obtain 
\begin{equation*}
	\begin{aligned}
	E_1(t,0)\,|{v}-{z}|_{{}_{L^2}} &\leq C \int_0^t N^2(s) (t-s)^{-1/2}\,E_2(s,t) \,E_1(0,s)\,ds\\
	& \quad + C\int_0^t \Omega^\varepsilon(\xi)(t-s)^{-1/2}\, E_2(s,t)\, E_1(s,0)\,ds \\
		& \quad + C\int_0^t  |{v}_0|^2_{{}_{L^2}} \, e^{2\Lambda^\e_2 s} \,
		(t-s)^{-1/2}\,E_2(s,t)\, E_1(s,0)\,ds\\
		&\leq C_1 N^2(t) \, E_1(0,t)+\left( |\Lambda^\e_2|^{-1/2}  |\Omega^\varepsilon|_{{}_{L^\infty}} E_1(t,0) + |\Lambda^\e_2|^{-3/4} |{v}_0|^2_{{}_{L^2}}\right),
	\end{aligned}
\end{equation*}
{where we used}
\begin{equation}\label{stimeintegrali}
\begin{aligned}
	&\int_0^t e^{(2\Lambda^\varepsilon_2-\Lambda_1^\varepsilon)s}\,ds
		\leq\frac{1}{\Lambda_2^\varepsilon}(e^{\Lambda^\varepsilon_2 t}-1)\leq  \frac{1}{|\Lambda_2^\varepsilon|},\\
	&\int_0^t (t-s)^{-1/2}\,E_2(s,t)\,ds
		\leq \int_0^t (t-s)^{-1/2}\,e^{\Lambda_2^\varepsilon\,(t-s)}\,ds
		\leq \frac{1}{|\Lambda_2^\varepsilon|^{1/2}}
\end{aligned}
\end{equation}
and $C_1$ and $C_2$ depend on $\Lambda_2^\varepsilon$. Hence, as soon as 
\begin{equation}\label{finalt}
4 C_1  \,{\Big(}|\Lambda^\e_2|^{-1/2}|\Omega^\varepsilon|_{{}_{L^\infty}} +  E_1(0,t)|\Lambda^\e_2|^{-3/4}|{v}_0|^2_{{}_{L^2}} \Big) <1
\end{equation}  
we obtain the following $L^2-$estimate for the difference ${v}-{z}$
\begin{equation}\label{estfin}
|{v}-{z}|_{{}_{L^2}} \leq  \left( |\Omega^\varepsilon|_{{}_{L^\infty}} + |\Lambda^\e_2|^{-3/4}\,|{v}_0|^2_{{}_{L^2}}  \, E_1(0,t) \right),
\end{equation}
where $E_1(0,t)$ behaves like $e^{-c t }$, since $\lambda_1^\e <0$, and $|\Lambda^\e_2|\sim \e^{-\alpha}$ for some $\alpha >0$. Condition \eqref{finalt} is a condition on the final time $T^\varepsilon$ and it can be rewritten as
\begin{equation*}
e^{\Lambda_1^\varepsilon t} \leq C\;\frac{1-|\Omega^\varepsilon|_{{}_{L^\infty}}}{ |\Lambda^\e_2|^{-3/4}\, |{v}_0|^2_{{}_{L^2}}}.
\end{equation*}
Hence, $T^\varepsilon$ can be chosen of order  $\left(\ln |\Lambda^\e_2|^{3/4}\right) |\Lambda^\varepsilon_1|^{-1}$, which is diverging to $+\infty$ as $|\Lambda^\e_1|^{-1}$ for $\e \to 0$.

\end{proof}

\subsection{{The case $\lambda^\e_1>0$ and the quadratic term $\mathcal Q^\e$ depending only on $v$}}

Since $\lambda_1^\e$ is positive, we can no longer use the theory of \cite{Pazy83}; indeed, we cannot state anymore that $\lambda_k^\varepsilon \leq -|\Lambda_1^\varepsilon| <0$ for all $k \geq 1$, so that we cannot construct a stable family of generators for $\mathcal L^\e_\xi$ as in the proof of Theorem \ref{T1}. In this case we can prove an estimate analogous to the one proved in Theorem \ref{T1bis}.

\vskip.15cm
\begin{theorem}\label{T2}
Let the hypotheses of Theorem \ref{T1bis} hold. Then, for every $t \leq T^\varepsilon$, there holds for the solution ${v}$
\begin{equation}\label{estfin2}
|{v}-{z}|_{{}_{L^2}}(t) \leq C \left( |\Omega^\varepsilon|_{{}_{L^\infty}} + |\Lambda_2^\e|^{-3/4} |{v}_0|^2_{{}_{L^2}}\right),
\end{equation}
where the function ${z}$ and the time $T^\varepsilon$ are defined as in Theorem \ref{T1bis}, and $\Lambda_2^\e:= \sup\limits_{\xi\in I} \lambda^\varepsilon_2(\xi) $. 
\end{theorem}

\begin{remark}{\rm

From hypothesis {\bf H2}, we know that $\lambda_2^\e \leq -c/\e^{\alpha}$, for some $\alpha \geq 0$. Hence, if $\alpha$ is strictly positive, \eqref{estfin2} assures the convergence to zero of the perturbation $v$ {as $\e\to0$}; on the contrary, if $\alpha=0$, we need to restrict our analysis to the case of small  (with respect to $\varepsilon$) initial data, i.e. we need to require ${v}_0 \in L^2(I)$ such that $|{v}_0|_{{}_{L^2}} \leq c \, \varepsilon$. This is a small deterioration of the estimate, consequence of the instability of the steady state, but it is however consistent with the case considered in \cite{CarrPego89}.
}
\end{remark}

\begin{proof}
 The proof is exactly the same as in Theorem \ref{T1bis};  recalling  \eqref{stimeintegrali}, we end up with the following estimate
 \begin{equation*}
|{v}-{z}|_{{}_{L^2}} \leq  \left( |\Omega^\varepsilon|_{{}_{L^\infty}} +  |\Lambda^\e_2|^{-3/4} \, | {v}_0|^2_{{}_{L^2}}  \, E_1(0,t) \right).
\end{equation*}
 Since $\lambda_1^\e$ is positive,  we can no longer assure the convergence to zero of the term $|{v}_0|^2_{{}_{L^2}}  \, E_1(0,t)$ in \eqref{estfin}; indeed, since $\lambda_1^\e \to 0$ as $\e \to 0$, $E_1(0,t)$ is converging to a constant for small $\e$ and large $t$. 
  We here use hypothesis {\bf H2} concerning the behavior of $\lambda_k^\e$, $k \geq 2$ and we recall that, if $\alpha=0$, we have the assumption $|{v}_0|_{{}_{L^2}} \leq c \, \varepsilon$.

\end{proof}

\begin{remark}{\rm
Let un consider the Allen-Cahn equation, i.e. problem \eqref{burgers} with $a(x)=1$ and $G(u,\partial_x u)= g(u)$, for some $g$ satisfying the following  assumptions: there exists a $C^2$ potential function $W: \R \to \R$ such that ${g}({u})= W'({u})$, and  we assume  $W$ to have  two distinct global minima $\pm {u}^*$ such that $W(\pm {u^*} )=0$;
in this case, is its well known (see, among others, \cite{AliFus08, Ste87}) that the only possible stable equilibrium solutions are constant in space and are given exactly by $\pm u^*$, while all the space dependent  steady states that present patterns of internal transition layers are unstable.
Additionally, in \cite{CarrPego89}, it is proven that the first eigenvalue of the linearized operator around an interface stationary solutions (which is unstable) is positive, but small in $\e$. Also, the nonlinear terms $\mathcal Q^\e$ depend only on the perturbation $v$, and can be estimated via the $L^2$- norm of $v$. This is an explicit and well known example where the equation exhibit a metastable behavior and Theorem \ref{T2} can be applied (see, for instance, \cite{Str13}).

Our guess is that, also in the case of a quasilinear second order term with $a(x)$ satisfying \eqref{ipoa}, a spectral analysis of the linearized operator can be performed in order to show that $\lambda^\e_1$ is positive and small in $\e$. 
}
\end{remark}

\subsection{{The case  $\mathcal Q^\e$ depending on $v$ and its space derivatives}}

\hskip1cm

{In order to prove an estimate for the perturbation $v$, we need an additional upper bound for the $L^2$-norm  of $\partial_x v$ and we have to consider $H^1$ initial data $v_0$.}

\begin{theorem}\label{T3}
Let hypotheses {\bf H1-2-H4.1} be satisfied and let us denote by $(\xi,v)$ the solution to the initial-value problem \eqref{NL}, with
\begin{equation*}
\xi(0)=\xi_0 \ \in \ I \quad {\rm and} \quad v(x,0)=v_0(x) \ \in \ H^1(I),
\end{equation*}
Then, for $\varepsilon$ sufficiently small, there exists a time $T^\varepsilon \geq 0$, such that, for any $t \leq T^\varepsilon$, the solution $v$ can be represented as
\begin{equation*}
v=z+R,
\end{equation*}
where $z$ is defined by
\begin{equation*}
	z(x,t):=\sum_{k\geq 2} v_k(0)\exp\left(\int_0^t \lambda^\varepsilon_k(\xi(\tau))\,d\tau\right)
		\,\varphi^\varepsilon_k(x;\xi(t)),
\end{equation*}
and the remainder $R$ satisfies the estimate
\begin{equation}\label{boundresto}
	|R|_{{}_{H^1}}\,\leq C\,
		\left\{ \varepsilon^\delta \, \exp\left( \int_0^t\lambda_1^\varepsilon(\xi(\tau))d\tau \right)|v_0|^2_{{}_{H^1}}+\varepsilon^{\alpha-\delta}+ |\Omega^\varepsilon|_{{}_{L^\infty}}\right\},
\end{equation}
for some constant $C>0$ and for some $\delta \in (0,\alpha)$, $\alpha > 0$. Furthermore, the final time $T^\varepsilon$ can be chosen of order $1/\e^{\gamma}$, for some $\gamma>0$.
\end{theorem}

\begin{remark}{\rm
{As will be clarified later,} the constant $\alpha$ in \eqref{boundresto} is exactly the one defined in hypothesis {\bf H2}. We do not consider the case $\alpha=0$ since it appears only when considering reaction diffusion systems, where the nonlinear term $\mathcal Q^\e$ depends only on the function $v$ (see, for example, \cite{Str13}). 
}
\end{remark}
\begin{remark}{\rm 
Since $\alpha >0$,  the final estimate for $v$ in Theorem \ref{T3} does not depend on the sign of the first eigenvalue $\lambda_1^\e$; indeed, the term $\varepsilon^\delta E_1(0,t) |v_0|_{{}_{H^1}}^2$ goes to zero as $\e \to 0$ whatever the sign of $\lambda_1^\e $ is. Hence, with respect to Theorem \ref{T2}, Theorem \ref{T3} holds for a general class of initial data $v_0 \in H^1$, not only the ones with $H^1$- norm small in $\e$. On the contrary, we also underline that the estimate \eqref{boundresto} for $v$ is weaker that the corresponding one obtained in \eqref{estfin} in Theorem \ref{T1bis}; indeed,   it states that the remainder $R$ tends to $0$ as $\varepsilon^\delta$ instead of $|\Omega^\e|_{{}_{L^\infty}}$, and we shall see in the following that this term behaves like $e^{-1/\e}$. Such deterioration  is a consequence of the necessity of estimating also the first order derivative.
}
\end{remark}

\begin{proof}[Proof of Theorem \ref{T3}]
Since the plan of the proof closely resemble the one used for proving Theorem \ref{T1bis}, we
propose here only the major modifications of the argument. In particular, the key point is how to handle the nonlinear terms.

Setting as usual
\begin{equation*}
	v(x,t)=\sum_{j} v_j(t)\,\varphi^\varepsilon_j(x,\xi(t)),
\end{equation*}
we obtain an infinite-dimensional differential system for the coefficients $v_j$
\begin{equation*}
	\frac{dv_k}{dt}=\lambda^\varepsilon_k(\xi)\,v_k
		+\langle \psi^\varepsilon_k,F_1\rangle+\langle \psi^\varepsilon_k,F_2\rangle,
\end{equation*}
where $F_1$ is defined as before as
\begin{equation*}
	F_1:=H^\varepsilon-\theta^\varepsilon \sum_{j}\Bigl(a_j+\sum_{\ell} b_{j\ell}\,v_\ell\Bigr)v_j,
\end{equation*}
with
\begin{equation*}
	a_j:=\langle \partial_{\xi} \psi^\varepsilon_1, \varphi^\varepsilon_j\rangle\,
			\partial_{\xi}U^{\varepsilon}+\partial_\xi \varphi^\varepsilon_j,
	\qquad
	b_{j\ell}:=\langle \partial_{\xi} \psi^\varepsilon_1, \varphi^\varepsilon_\ell\rangle
			\,\partial_\xi \varphi^\varepsilon_j.
\end{equation*}
The term $F_2$ comes out from the higher order terms $\rho^\varepsilon$ and $\mathcal R^\varepsilon$ and has the following expression
\begin{equation*}
F_2:= \mathcal Q^\varepsilon- \left(  \sum_j  \partial_\xi \varphi_j^\varepsilon v_j+\partial_\xi U^\varepsilon\right) \left\{ \frac{\langle \psi^\varepsilon_1,\mathcal Q^\varepsilon \rangle}{1-\langle\partial_\xi \psi^\varepsilon_1,v \rangle}-\theta^\varepsilon\frac{\langle \partial_\xi \psi^\varepsilon_1,v \rangle^2}{1-\langle\partial_\xi \psi^\varepsilon_1,v \rangle} \right\}.
\end{equation*}
Moreover, we have
\begin{equation*}
|\langle \psi^\varepsilon_k, F_2 \rangle| \leq (1+|\Omega^\varepsilon|_{{}_{L^\infty}}) |v|^2_{{}_{L^2}}+ C |v|^2_{{}_{H^1}},
\end{equation*}
so that, setting as before
\begin{equation*}
	E_k(s,t):=\exp\left( \int_s^t \lambda_k^\varepsilon(\xi(\tau))d\tau\right)
\end{equation*}
and since $v_1=0$,  we have the following expression for the coefficients $v_k$, $k\geq 2$
\begin{equation*}
v_k(t)= v_k(0) E_k(0,t)+ \int_0^t \left\{ \right \langle \psi^\varepsilon_k,F_1\rangle +\langle \psi^\varepsilon_k,F_2\rangle \} E_k(s,t) \, ds.
\end{equation*}
By introducing the function
\begin{equation*}
	z(x,t):=\sum_{k\geq 2} v_k(0)\,E_k(0,t)\,\varphi^\varepsilon_k(x;\xi(t)),
\end{equation*}
we end up with the following estimate for the $L^2$-norm of the difference $v-z$
\begin{equation*}
|v-z|_{{}_{L^2}} \leq \sum_{k \geq 2} \int_0^t \left( \Omega^\e(\xi)(1+|v|^2_{{}_{L^2}}) + |v|^2_{{}_{H^1}}\right) E_k(s,t) \, ds,
\end{equation*}
that is
\begin{equation*}
\begin{aligned}
|v-z|_{{}_{L^2}} \leq& \ C \int_0^t \Omega^\varepsilon(\xi) (t-s)^{-1/2} E_2(s,t)  ds \\
&+ \int_0^t \left( |v-z|^2_{{}_{H^1}}+ |z|^2_{{}_{H^1}}\right) (t-s)^{-1/2} E_2(s,t)  ds,
\end{aligned}
\end{equation*}
where we used
 \begin{equation*}
	\sum_{k\geq 2} E_k(s,t)
		\leq C\,(t-s)^{-1/2}\,E_2(s,t).
\end{equation*}
Now we need to differentiate with respect to $x$ the equation for $v$ in order to obtain an estimate for $|\partial_x(v-z)|_{{}_{L^2}}$. By setting $y=\partial_x v$, we obtain
\begin{equation*}
\partial_t y= \mathcal L^\varepsilon_\xi y + \bar{\mathcal  M}^\varepsilon_\xi v+ \partial_x \left( d\mathcal P^\e[U^\e]\right) v + \partial_xH^\varepsilon(x,\xi)+\partial_x \mathcal R^\varepsilon[v,\xi],
\end{equation*}
where
\begin{equation*}
\bar{\mathcal M}^\varepsilon_\xi v:=-\partial_{ x \xi}U^{\varepsilon}(\cdot;\xi)
			\,\theta^\varepsilon(\xi)\,\langle\partial_{\xi} \psi^\varepsilon_1, v \rangle \quad {\rm and} \quad  \partial_x\left(d \mathcal P^\e[U^\e]\right) v:= \e \, a'(x) \, \partial_x v-dG[U^\e] v.
\end{equation*}
Hence, by setting as usual
\begin{equation*}
y(x,t)=\sum_j y_j(t) \, \varphi_j^\varepsilon (x,\xi(t)),
\end{equation*}
we have
\begin{equation*}
\frac{dy_k}{dt}=\lambda^\varepsilon_k(\xi)\,y_k
		+\langle \psi^\varepsilon_k,F^*\rangle+\langle \psi^\varepsilon_k ,\partial_x\left( d\mathcal P^\e[U^\e] v\right)\rangle +\langle \psi^\varepsilon_k,\partial_x \mathcal R^\varepsilon \rangle,
		\end{equation*}
where
\begin{equation*}
F^*:= \partial_x H^\varepsilon -\sum_j v_j \left \{ \theta^\varepsilon\left[ \partial_{ x \xi} U^\varepsilon \langle \partial_\xi \psi^\varepsilon_1,\varphi_j^\varepsilon \rangle + \partial_\xi \varphi^\varepsilon_j\left( 1 + \sum_{\ell}v_{\ell} \langle \partial_\xi \psi^\varepsilon_1,\varphi_{\ell} \rangle \right)   \right] -\partial_\xi \varphi^\varepsilon_j \rho^\varepsilon\right \}.
\end{equation*}
Moreover, by integrating by parts,  for some $m >0$ there hold
\begin{equation*}
|\langle \psi^\varepsilon_k, \partial_x \mathcal R^\varepsilon \rangle| \leq C |v|^2_{{}_{H^1}}, \quad \langle \psi^\varepsilon_k ,\partial_x\left( d\mathcal P^\e[U^\e] v\right)\rangle \leq \varepsilon^m \left( |dG|_{{}_{L^\infty}}^2+ \e\,  |a'|_{{}_{L^\infty}}^2 \right)+ \frac{1}{\varepsilon^m}\,|v|^2_{{}_{H^1}}.
\end{equation*}
Because of the assumptions on $a$ and $G$, there holds $|dG|_{{}_{L^\infty}}^2+ |a'|_{{}_{L^\infty}}^2 <c$ for some positive constant, so that, by integrating in time and by summing on $k$, we end up with
\begin{equation*}
\begin{aligned}
|y-\partial_x z|_{{}_{L^2}} \leq&\, C \int_0^t \left \{ \Omega^\varepsilon(\xi) (1+ |v|^2_{{}_{L^2}}) +\left(1+\frac{1}{\varepsilon^m}\right)\, |v|^2_{{}_{H^1}} + \varepsilon^m  \right \} \, E_1(s,t) ds \\
&+C \int_0^t \left \{ \Omega^\varepsilon(\xi) (1+ |v|^2_{{}_{L^2}}) +\left(1+\frac{1}{\varepsilon^m}\right)\, |v|^2_{{}_{H^1}} + \varepsilon^m \right \} \, \sum_{k \geq 2} E_k(s,t) ds.
\end{aligned}
\end{equation*}
Now, given $n>0$, let us set
\begin{equation*}
	N(t):= \frac{1}{\varepsilon^n} \sup_{s\in[0,t]} |v-z|_{{}_{H^1}}\,E_1(s,0),
\end{equation*} 
so that we have
\begin{equation}\label{est1}
\begin{aligned}
\frac{1}{\varepsilon^n}E_1(t,0)|v-z|_{{}_{L^2}} \leq& \ C \int_0^t \frac{\Omega^\varepsilon(\xi)}{\varepsilon^n} (t-s)^{-1/2} E_2(s,t) E_1(s,0) ds \\
&+ \int_0^t \frac{1}{\varepsilon^n}\left( |v-z|^2_{{}_{H^1}}+ |z|^2_{{}_{H^1}}\right) (t-s)^{-1/2} E_2(s,t) E_1(s,0) ds.
\end{aligned}
\end{equation}
and
\begin{equation}\label{est2}
\begin{aligned}
\frac{1}{\varepsilon^n}&E_1(t,0)|y-\partial_x z|_{{}_{L^2}} \leq  
C \int_0^t \frac{\Omega^\varepsilon(\xi)}{\varepsilon^n}  \left \{E_1(s,0)+(t-s)^{-1/2} E_s(s,t)\, E_1(s,0)\right \} \,  ds\\
& +C \int_0^t \left \{ \left(\frac{1}{\e^n}+\frac{1}{\e^{n+m} } \right)\,\left( |v-z|^2_{{}_{H^1}} + |z|^2_{{}_{H^1}}\right)+\frac{1}{\varepsilon^{n-m}}  \right \} \, E_1(s,0) ds \\
& +C \int_0^t\left \{ \left(\frac{1}{\e^n}+\frac{1}{\e^{n+m}} \right)\!\!\left( |v-z|^2_{{}_{H^1}} \!+\! |z|^2_{{}_{H^1}}\right)\!+\!\frac{1}{\e^{n-m}}  \right \} (t-s)^{-1/2} E_s(s,t) E_1(s,0) ds. 
\end{aligned}
\end{equation}
By summing \eqref{est1} and \eqref{est2} and since there hold
\begin{align*}
	&\int_0^t e^{(2\Lambda^\varepsilon_2-\Lambda_1^\varepsilon)s}\,ds \leq  \frac{1}{|\Lambda_2^\varepsilon|},	\\
	&\int_0^t (t-s)^{-1/2}\,E_2(s,t)\,ds
		\leq \int_0^t (t-s)^{-1/2}\,e^{\Lambda_2^\varepsilon\,(t-s)}\,ds
		\leq \frac{1}{|\Lambda_2^\varepsilon|^{1/2}},
\end{align*}
we end up with the estimate $N(t) \leq A N^2(t)+B$, with
 \begin{equation*}
		\left\{\begin{aligned}
		A&:= \varepsilon^{-n-m} \, E_1(0,t) (t + |\Lambda_2^\varepsilon|^{-1/2}) ,\\
		B&:=C|\Omega^\varepsilon|_{{}_{L^\infty}}E_1(t,0) \left( t + |\Lambda_2^\varepsilon|^{-1/2} \right) \\
		& \quad+ \varepsilon^{-n-m} |\Lambda_2^\varepsilon|^{-1} |v_0|^2_{{}_{H^1}}+ \varepsilon^{m-n}  E_1(t,0)\, (t +|\Lambda_2^\varepsilon|^{-1/2}).
		\end{aligned}\right.
\end{equation*}
We now use the assumptions on the behavior of $\lambda_k^\e$, $k \geq 2$; since there holds $|\Lambda_2^\varepsilon| \sim \e^{-\alpha}$ for some $\alpha > 0$, if we require $m < \alpha$, for all $n>0$ there holds $N(t)<  B$ that is
\begin{equation}\label{finnonlin}
|v-z|_{{}_{H^1}} \leq C |\Omega^\varepsilon|_{{}_{L^\infty}} + \Bigl( \varepsilon^{\alpha-m} |v_0|^2_{{}_{H^1}}E_1(0,t)+\e^m\Bigr).
\end{equation}
Precisely,  we can choose $m=\alpha-\delta$, for some $\delta \in (0,\alpha)$. Finally, providing  $m>n$, we can choose the final time $T^\varepsilon$ of order $\mathcal O(\e^{-\gamma})$ for some $0<\gamma<1$. Now the proof is completed.
\end{proof}

\begin{remark}{\rm
The  proofs of Theorems \ref{T1bis} and \ref{T3} can be easily extend to the case  $v \in [H^1(I)]^n$ with only minor changes; this is meaningful in light of a possible application of these results in the case of a system of equations of the form \eqref{burgers}. We point out that, in the case $u\in [H^1(I)]^n$, in hypothesis {\bf H2} one should consider the chance of having complex eigenvalues; however, the first eigenvalue has to be real in order to observe a metastable behavior.

}
\end{remark}

\subsection{The slow motion of the shock layer}

\vskip0.5cm
The estimates for the perturbation $v$ obtained in the previous section  can be used to decouple the system \eqref{NL} in order to obtain an equation of motion for the parameter $\xi(t)$. Indeed, both \eqref{stimafinalev}, \eqref{estfin} and \eqref{finnonlin} imply that, for small $\e$, the perturbation $v$ converges to zero as $t \to \infty$. This preludes the following result. 

\begin{proposition}\label{SMshock}
Let either the hypothesis of Theorem \ref{T1}, \ref{T2} or \ref{T3} be satisfied. Let us also assume that
\begin{equation}\label{ipoteta}
	(\xi- \xi^*)\,\theta^\varepsilon(\xi)<0\quad\textrm{ for any } \xi\in I,\,\xi\neq \xi^*
	\qquad\textrm{ and }\qquad
	{\theta^\varepsilon}'( \xi^*)<0.
\end{equation}
Then, for $\varepsilon$ sufficiently small, the solution $\xi(t)$ converges exponentially to $\xi^*$ as $t\to+\infty$.

\end{proposition}

\begin{proof}
Again, we need to divide the proof in two main parts. At first, we consider the case where the nonlinear higher order terms in the equation for the perturbation $v$ can be estimated via the $L^2$ norm of $v$ itself. Since either \eqref{estfin} or \eqref{estfin2} hold, we get
\begin{equation*}
\frac{d\xi}{dt}=\theta^\varepsilon(\xi)(1+r){+\rho^\e} \quad \text{with} \quad |r| \leq  C\left( |\Omega^\varepsilon|_{{}_{L^\infty}} +|\Lambda_2^{\varepsilon}|^{-3/2}|v_0|_{{}_{L^2}} \right)\quad \text{{and}} \quad{|\rho^\e{|\leq C|\Lambda_2^{\varepsilon}|^{-3/2}|v_0|^2_{{}_{L^2}}},}
\end{equation*}
By a method of separation of variables, and since $\theta^\varepsilon(\xi) \sim {\theta^\varepsilon}'(\xi^*) (\xi -\xi^*) $, we end up with
\begin{equation*}
\xi(t) \sim \xi^*+ {(\xi_0-\xi^*)} e^{-\beta^\varepsilon t}+ C|\Lambda_2^{\varepsilon}|^{-3/2}|v_0|^2_{{}_{L^2}} \left(1- e^{-\beta^\e t}\right),
\end{equation*}
where $\beta^\varepsilon:= -{\theta^\varepsilon}'(\xi^*)(1+|r|) \sim -{\theta^\varepsilon}'(\xi^*)$ for small $\e$,
and $\xi^*$ corresponds to the asymptotic location for the parameter $\xi$ (we recall that $U^\varepsilon(x,\xi^*)$ is an exact steady state for the system). Hence, for small $\e$ and large $t$, $\xi$ is converging to $\xi^*$ with exponential rate.
\vskip0.2cm
When the nonlinear term $\mathcal Q^\e$ also depends on the space derivative of $v$, we use \eqref{finnonlin} and we get
$
\frac{d\xi}{dt}= \theta^\varepsilon(\xi)(1+r)+ \rho^\varepsilon,
$
where both $r$ and $\rho^\e$ can be estimated via $ (\varepsilon^m + \varepsilon^{\alpha-m}) + |\Omega^\e|_{{}_{L^\infty}}$.
Again, since both $r$ and $\rho^\e$ go to zero as $\e \to 0$,  in the vanishing viscosity limit there holds
\begin{equation*}
|\xi-\xi^*| \leq{ |\xi_0-\xi^*|}e^{-\beta^\varepsilon t}, \qquad \beta^\e \sim -{\theta^\varepsilon}'(\xi^*),
\end{equation*}
showing the exponentially (slow) motion of the position of the interface towards its equilibrium location $\xi^*$.
\end{proof}

\section{Application to quasilinear viscous scalar conservation laws}

In this Section we mean to apply the general theory previously developed to the specific example of a quasilinear conservation law, i.e. we consider the following  initial-boundary-value  problem 
\begin{equation}\label{Burgers}
 \left\{\begin{aligned}
& \partial_t u = \varepsilon \partial_x\left( a(x) \partial_x u \right)-\partial_x f(u),  &\qquad &x \in I, \ t \geq 0, \\
& u(\pm \ell,t)=u_{\pm}, &\qquad &t \geq 0,\\
& u(x,0)=u_0(x), &\qquad &x \in I,
  \end{aligned}\right.
  \end{equation}
where the flux function $f\in C^2(I)$ satisfies the standard hypotheses  
\begin{equation}\label{ipof}
f''(u) \geq c_0 >0, \quad f'(u_+) < 0 <f'(u_-), \quad f(u_+)=f(u_-).
\end{equation}
This is a well know and simplified prototype of problem \eqref{burgers} in the case of a forcing term $G(u,\partial_x u)$ depending both on $u$ and on its space derivative.
\vskip0.2cm
Formally, in the vanishing viscosity  limit $\varepsilon\to 0^+$, equation \eqref{Burgers}
reduces to a first-order quasi-linear hyperbolic problem on the form
\begin{equation}\label{unviscburgers}
	\partial_t u + \partial_x f(u)=0, \qquad u(x,0)=u_0(x)
\end{equation}
complemented with boundary conditions
\begin{equation}\label{conbordo}
u(-\ell,t)=u_- \quad {\rm and} \quad u(\ell,t)=u_+.
\end{equation}
The standard setting of solutions to \eqref{unviscburgers} is well known, and it is the one given by the {\it entropy formulation}.
Hence, we may have solutions with discontinuities, which propagate with a speed  $s$ dictated by the well known  \textrm{{\sf Rankine--Hugoniot relation}}
\begin{equation*}
 s\ldbrack u\rdbrack=\ldbrack f(u)\rdbrack,
 \end{equation*}
and  that satisfy appropriate {\sf entropy conditions}. Assumptions \eqref{ipof} on the flux function $f$ guarantee that the jump from the value $u_-$ to the value $u_+$ is admissible if and only if $u_->u_+$, and that its speed of propagation $s$ is equal to zero. 

\begin{figure}
\centering
\includegraphics[width=10cm,height=7cm]{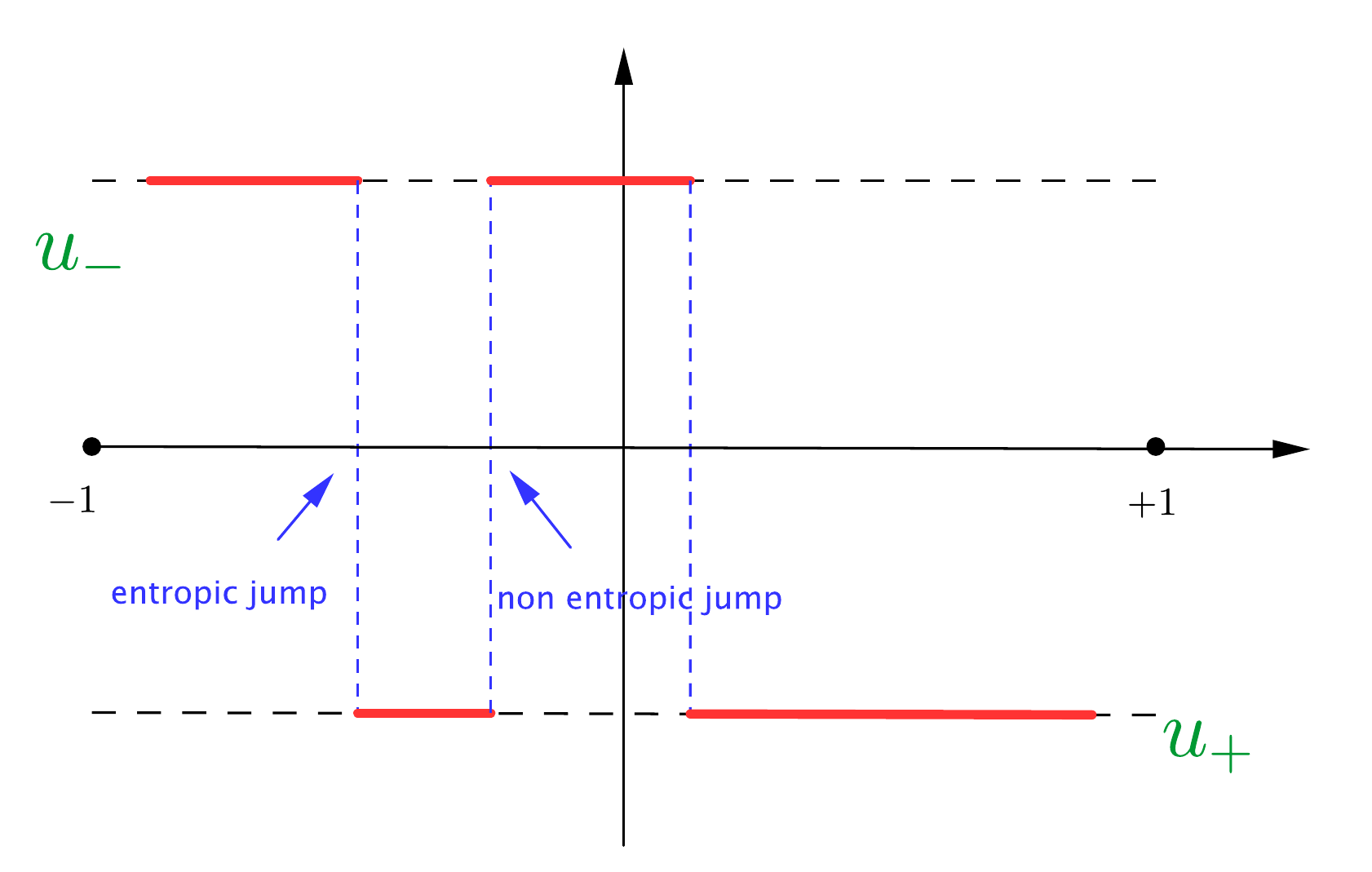}
\caption{\small{ Picture of the steady state to \eqref{unviscburgers}; because of the entropy conditions, the only jumps admitted are the ones from a value $u_->u_+$.}}
\end{figure}

In this case, equation \eqref{unviscburgers} admits a large class of stationary solutions satisfying the boundary conditions, given by all that piecewise constant functions in the form
\begin{equation*}
 u(x) =  \left\{\begin{aligned}
&u_- \quad x \in (-\ell,x_0), \\
&u_+ \quad x \in (x_0,\ell),
   \end{aligned}\right.
\end{equation*} 
where $x_0 \in I$ is a certain point in the interval (see Figure 1). Hence, given $\xi \in (-\ell,\ell)$, we can construct a one-parameter family $\{ U_{{}_{\rm hyp}}(\cdot;\xi)\}$ of steady states, parametrized by $\xi$ that represents the location of the jump, and given by
\begin{equation*}
 U_{{}_{\rm hyp}}(x;\xi)= u_- \chi_{(-\ell,\xi)}(x)+ u_+ \chi_{(\xi,\ell)}(x),
\end{equation*} 
where $\chi_I$ denotes the characteristic function of the interval $I$.

For the initial-boundary value problem \eqref{unviscburgers}-\eqref{conbordo}, it is possible to prove that, starting from an initial datum $u_0$ with bounded variation, every entropy solution converges in finite time to an element of the family  $\{ U_{{}_{\rm hyp}}(\cdot;\xi)\}$. 
\vskip.15cm

For $\varepsilon>0$, the situation is very different. In this case, because of the diffusive term, no discontinuities are admitted, and there is a drastic reduction of the number of stationary solutions; indeed, it is possible to prove (see Section 4.1) that there exists only one single steady state satisfying 
\begin{equation*}
 \left\{\begin{aligned}
&  \varepsilon \partial_x\left( a(x) \partial_x u \right)=\partial_x f(u),   \\
& u(\pm \ell)=u_{\pm}.\\
  \end{aligned}\right.
\end{equation*}
Such solution, denoted here by 
$\bar U^\varepsilon(x)$,
converges pointwise in the limit $\varepsilon\to 0^+$ to a specific element  $U_{{}_{\textrm{hyp}}}(\cdot;\bar\xi)$ of the family 
$\{U_{{}_{\textrm{hyp}}}(\cdot;\xi)\}$, for some $\xi^* \in I$.

Finally, the single steady state $\bar U^\varepsilon$ is asymptotically stable (for more details see the spectral analysis performed in Section 4.2), i.e. starting from an initial datum close to the equilibrium configuration, the time dependent solution  approaches the steady state for $t\to +\infty$.
\vskip0.5cm

\subsection{The stationary problem}

In order to apply to problem \eqref{Burgers}  the general theory developed in Section 3, the first step is the construction of the family of approximate steady states defined in hypothesis {\bf H1}.

Let us then consider the stationary problem for \eqref{Burgers}, that is
\begin{equation}\label{BurgStat}
 \left\{\begin{aligned}
		\varepsilon\,\partial_x\left( a(x)\partial_x u \right) &= \partial_x f(u)\\
 		u(-\ell)=u_-,	\quad & u(\ell)=u_+	 \\
 	 \end{aligned}\right.
\end{equation}
for $x\in (-\ell,\ell)$. This problem has been extensively studied in the case $a(x)=1$ (see, for instance \cite{KreiKrei86}). Let us begin our study with some explicit examples.

\begin{example}{The Burgers equation.}\label{ex1}
{\rm
In the case of Burgers equation, i.e. $f(u)=u^2/2$ and $a(x)\equiv1$, the value $u_+$ coincides with $-u_-$, and the stationary problem reads
\begin{equation*}
\e\partial_x^2 u = \frac{1}{2}\partial_x (u^2), \quad u(\pm \ell)=\mp u^*.
\end{equation*}
By integrating, we obtain an explicit expression for the unique steady state, that is
\begin{equation*}
	\bar U^\varepsilon(x)
		=-\kappa\,\tanh\left(\frac{\kappa\,x}{2\,\varepsilon}\right),
\end{equation*}
where $\kappa=\kappa(\varepsilon,\ell,{u^*})$ is univocally determined once the boundary conditions are imposed.
\vskip0.3cm
Following the general approach introduced in the previous sections, we want to construct  the family of approximate steady states $\{ U^\varepsilon(x;\xi)\}$. There are several choices to built-up such a family (for instance, the approach of \cite{deGrKara98} based on the existence of traveling waves on the whole line). 
We here recall the approach of \cite{MS}, and we consider a function obtained by matching two different steady states satisfying, respectively, the left 
and the right boundary condition together with the request $U^\varepsilon(\xi)=0$; in formulas,
\begin{equation}\label{defUe}
	U^{\varepsilon}(x;\xi)=\left\{\begin{aligned}
		&\kappa_- \tanh \left(\kappa_-(\xi-x)/2\varepsilon\right) &\qquad &{\rm in}\quad (-\ell,\xi) \\
		&\kappa_+ \tanh\left(\kappa_+(\xi-x)/2\varepsilon\right) &\qquad &{\rm in}\quad (\xi,\ell),
           \end{aligned}\right.
\end{equation}
where $\kappa_\pm$ are chosen so that the boundary conditions are satisfied.
By direct substitution we obtain the identity
\begin{equation*}
	{\mathcal P}^\varepsilon[U^{\varepsilon}(\cdot;\xi)]
		=\varepsilon\ldbrack \partial_x U^{\varepsilon}\rdbrack_{{}_{x=\xi}}\delta_{{}_{x=\xi}}
\end{equation*}
in the sense of distributions, where $\delta_{x=\xi}$ the usual Dirac's delta distribution centered 
in $x=\xi$. Going further, we have
\begin{equation*}
	\ldbrack \partial_x U^{\varepsilon}\rdbrack_{{}_{x=\xi}}
		=\frac{1}{2\varepsilon}(\kappa_--\kappa_+)(\kappa_-+\kappa_+).
\end{equation*}
In order to determine the behavior of $\mathcal P^\varepsilon[U^\varepsilon(\cdot;\xi)]$ for small $\varepsilon$, we need an asymptotic description of the values $k_{\pm}$.  Let us set $k_{\pm}:=\mp u_{\pm}(1+h_{\pm})$ and $\Delta_{\pm}:=\ell \mp \xi$. Then it results
\begin{equation*}
\tanh{\left( \mp \frac{u_{\pm}\Delta_{\pm}}{2 \varepsilon}(1+h_{\pm})\right)}=\frac{1}{1+h_{\pm}}.
\end{equation*}
Therefore, the values $h_{\pm}$ {can be chosen} both positive and then
    \begin{equation*}
{0\leq}\tanh{\left( \mp \frac{u_{\pm}\Delta_{\pm}}{2\varepsilon}\right)} {<} \frac{1}{1+h_{\pm}},
\end{equation*}
that gives the asymptotic representation
    \begin{equation*}
h_{\pm} {<}\frac{1}{\tanh{(\mp u_{\pm}\Delta_{\pm}/2\varepsilon)}}-1= \frac{2}{e^{\mp u_{\pm}\Delta_{\pm}/\varepsilon}-1}= 2e^{\pm u_{\pm}\Delta_{\pm}/\varepsilon}+ l.o.t
\end{equation*}
where {\it  l.o.t.} denotes lower order terms. In particular, since $u_\pm =\mp u^*$ for some $u^*>0$, there holds
\begin{equation*}
h_{\pm}  \sim e^{-u^*\Delta_{\pm}/\varepsilon}
\end{equation*}
Finally, since
 \begin{equation*}
[\![\partial_x U^\varepsilon]\!]_{x=\xi}=\frac{1}{2\varepsilon}(k_--k_+)(k_-+k_+){\sim}\frac{u^2_*}{\varepsilon}(h_--h_+),
\end{equation*}
we end up with
    \begin{equation*}
[\![\partial_x U^\varepsilon]\!]_{x=\xi}\,{\sim}\, \frac{u^2_*}{\varepsilon}(e^{-u_*(\ell+\xi)/\varepsilon}-e^{-u_*(\ell-\xi)/\varepsilon})\sim C\,\xi\,e^{-C/\varepsilon},
\end{equation*}
showing that this term is exponentially small for $\varepsilon \to 0$ and it is null when $\xi=0$, that corresponds to the equilibrium location of the shock when $f(u)=u^2/2$. {Hence} we have the following asymptotic expression for the term $\Omega^\e$ defined in \eqref{defomegaeps}
\begin{equation*}
\Omega^\e(\xi) \sim e^{-u_*(\ell+\xi)/\varepsilon}-e^{-u_*(\ell-\xi)/\varepsilon},
\end{equation*} 
that shows that hypothesis {\bf H1} is satisfied in this specific case.
}
\end{example}

\begin{example}{The quasilinear viscous Burgers equation.}
{\rm
In the case of a generic function $a(x)$ satisfying \eqref{ipoa}, the expression for the unique steady state is given by
\begin{equation}\label{defstazionaria}
	\bar U^\varepsilon(x)
		=-\kappa\,\tanh\left(\frac{\kappa\,\left[b(x)-b(l)/2\right]}{2\,\varepsilon}\right)
\end{equation}
where 
\begin{equation*}
b(x):= \int_{-\ell}^x  \frac{1}{a(t)} \, dt,
\end{equation*}
is known to exist because of the assumption \eqref{ipoa}. Let us observe that, if we call $N$ the number of zeroes of $b(x)-b(l)/2$, the number of the layers of the steady state $\bar U^\varepsilon(x)$ defined as in \eqref{defstazionaria} is exactly $N$; hence,   since the only jumps admitted are the ones from a a value  $u_1$ to a value $u_2 <u_1$, and since the boundary conditions are such that $u_- > u_+$, there exists a unique solution to \eqref{BurgStat} if and only if the function $b(x){-b(l)/2}$ has a single zero located at some point $x^* \in I$. In this case, the steady state $\bar U^\e$ will be approximately given by
\begin{equation*}
	\bar U^\varepsilon(x)
		\sim-\kappa\,\tanh\left(\frac{\kappa\,(x-x^*)}{2\,\varepsilon}\right).
\end{equation*}
The family of approximate steady is constructed as in Example \ref{ex1}, recalling that now $\bar U^\e$ is null when $x=x^*$. Hence
\begin{equation*}
	U^{\varepsilon}(x;\xi)=\left\{\begin{aligned}
		&\kappa_- \tanh \left(\kappa_-[(\xi-x^*)-x]/2\varepsilon\right) &\qquad &{\rm in}\quad (-\ell,\xi-x^*) \\
		&\kappa_+ \tanh\left(\kappa_+[(\xi-x^*)-x]/2\varepsilon\right) &\qquad &{\rm in}\quad (\xi-x^*,\ell),
           \end{aligned}\right.
\end{equation*}
In particular, there holds
\begin{equation*}
	{\mathcal P}^\varepsilon[U^{\varepsilon}(\cdot;\xi)]
		=\varepsilon\ldbrack \partial_x U^{\varepsilon}\rdbrack_{{}_{x=\xi-x^*}}\delta_{{}_{x=\xi-x^*}},
\end{equation*}
so that
\begin{equation*}
\Omega^\e(\xi) \sim e^{-u_*(\ell+\xi-x^*)/\varepsilon}-e^{-u_*(\ell-\xi+x^*)/\varepsilon},
\end{equation*} 
which is null when $\xi=x^*$, corresponding to the location of the {interface of the} exact steady state for the problem.
}
\end{example}

\vskip0.2cm
In the general case, if the  flux function $f(x)$ satisfies hypotheses \eqref{ipof}, solutions to \eqref{BurgStat} can be found implicitly  via de formula
\begin{equation}\label{stazimplicita}
	\int_{u(x)}^{u_-} \frac{ds}{\kappa-f(s)}=\frac{1}{\varepsilon}\int_{-\ell}^x a^{-1}(x) \,dx
\end{equation}
where $\kappa\in(f(u_\pm),+\infty)$ is such that
\begin{equation*}
	\Phi(\kappa):=\int_{u_+}^{u_-} \frac{ds}{\kappa-f(s)}=\frac{1}{\varepsilon}\int_{-\ell}^{\ell}a^{-1}(x) \,dx
\end{equation*}
Assumptions \eqref{ipof} on the flux $f$ imply that $\Phi$ is strictly
decreasing and such that 
\begin{equation*}
	\lim_{\kappa\to f(u_\pm)^+} \Phi(\kappa)=+\infty,\qquad
	\lim_{\kappa\to +\infty} \Phi(\kappa)=0.
\end{equation*}
Therefore, for any $\ell>0$, there exists a unique solution to \eqref{stazimplicita} satisfying the boundary conditions.

\vskip0.2cm
As in the previous example, the family $U^\e$ is build up by matching at $\xi{-x^*} \in (-\ell, \ell)$ two different steady stated $U^\e_-$ and $U^\e_+$, solutions to the stationary equations in $(-\ell,\xi{-x^*})$ and $(\xi{-x^*},\ell)$ respectively; hence we have
\begin{equation}\label{soluzapprox}
	U^{\varepsilon}(x;\xi)=\left\{\begin{aligned}
		&U^\varepsilon_-(x;\xi) 	&\qquad &	-\ell<x<\xi-x^*<\ell \\
		&U^\varepsilon_+(x;\xi)	&\qquad &-\ell<\xi-x^*<x<\ell ,
           \end{aligned}\right.
\end{equation}
and 
\begin{equation*}
	U^\varepsilon_-(-\ell;\xi)=u_-,\quad U^\varepsilon_-(\xi-x^*;\xi)=u_\ast
	\qquad\textrm{and}\qquad
	U^\varepsilon_+(\xi-x^*;\xi)=u_\ast,\quad U^\varepsilon_+(\ell;\xi)=u_+,
\end{equation*}
where $u^*$ is such that $f'(u^*)=0$ and $x^*$ is the unique zero of the exact steady state $\bar U^\e(x)$. Moreover, without loss of generality, we assume that $f(u^*)=0$. Note that $u^*=0$ in the case of a Burgers flux $f(u)=u^2/2$; this is consistent with the interpretation of $\xi$ as the location of the interface of $u$.  Once again the error $\mathcal P^\e[U^\e]$ is given by
\begin{equation*}
	{\mathcal P}^\varepsilon[U^{\varepsilon}(\cdot;\xi)]
		=\varepsilon\ldbrack \partial_x U^{\varepsilon}\rdbrack_{{}_{x=\xi-x^*}}\delta_{{}_{x=\xi-x^*}}.
\end{equation*}
Now $U^\varepsilon_-(x;\xi)$ and $U^\varepsilon_+(x;\xi)$ are implicitly given by
\begin{gather*}
\varepsilon\int^{u_-}_{U^\varepsilon_-(x;\xi)}\frac{ds}{k_--f(s)}=\int^x_{-l}a^{-1}(s)ds\quad x\in(-l,\xi-x^*)\\
\varepsilon\int^{u^*}_{U^\varepsilon_+(x;\xi)}\frac{ds}{k_+-f(s)}=\int^x_{\xi-x^*}a^{-1}(s)ds\quad x\in(x-x^*,l)
\end{gather*}
where $k_\pm$ are chosen such that the boundary conditions are satisfied. Then
$$\partial_x U^\varepsilon(x;\xi)=\begin{cases}
\displaystyle{\frac{a^{-1}(x)(f(U^\varepsilon_-(x;\xi))-k_-)}{\varepsilon}}\quad x\in(-l,\xi-x^*)\\
\displaystyle{\frac{a^{-1}(x)(f(U^\varepsilon_+(x;\xi))-k_+)}{\varepsilon}}\quad x\in(\xi-x^*,l),
\end{cases}$$
and we immediately obtain  
\begin{gather*}\ldbrack \partial_x U^{\varepsilon}\rdbrack_{{}_{x=\xi-x^*}}=\frac{1}{\varepsilon a(\xi-x^*)}\left(f(U^\varepsilon_+(\xi-x^*;\xi))-k_+-f(U^\varepsilon_-(\xi-x^*;\xi))+k_-\right)=\\
=\frac{1}{\varepsilon a(\xi-x^*)}\left(k_--k_+\right)\leq\frac{1}{\alpha\varepsilon}|k_--k_+|.
\end{gather*}
Using the following bounds on $f$
\begin{equation*}
	\begin{aligned}
		f(u_\pm)+f'(u_+)(u-u_+) &\leq
			f(u)\leq \frac{f(u_\pm)}{u_\ast-u_+}(u_\ast-u)	&\qquad	&u\in[u_+,u_\ast],\\
		f(u_\pm)-f'(u_-)(u_--u) &\leq
			f(u)\leq \frac{f(u_\pm)}{u_--u_\ast}(u-u_\ast)	&\qquad	&u\in[u_\ast,u_-]
	\end{aligned}
\end{equation*}
we can proceed as in \cite{MS} to estimate $|k_--k_+|$, proving that, for any $\delta\in(0,\ell)$, there exists $C>0$, independent on $\varepsilon$,
such that
\begin{equation*}
	\varepsilon\bigl|\ldbrack \partial_x U^{\varepsilon}\rdbrack_{{}_{x=\xi-x^*}}\bigr|
	\leq \,e^{-C/\varepsilon}
	\qquad\qquad \forall\,\xi-x^*\in(-\ell+\delta,\ell-\delta),
\end{equation*}
showing that hypothesis {\bf H1} is satisfied.

\subsection{Spectral analysis}

We mean to analyze the spectrum of the operator 
$$\mathcal L^\e_\xi v:= \e \partial_x \left( a(x) \partial_x v\right)-\partial_x \left(f'(U^\e) \, v \right), $$
obtained from the linearization of \eqref{burgers} around an element of the family \eqref{soluzapprox}; in particular, we obtain a precise distributions of the eigenvalues $\lambda_k^\e(\xi)$.

This analysis is needed in order to show that the general theory previously developed is applicable in the specific case of quasilinear viscous scalar conservation laws; more precisely, we show that the hypothesis {\bf H2} concerning the distribution of the eigenvalues of the linearized operator is satisfied in a concrete situation.

The eigenvalue problem reads
\begin{equation*}
\e \partial_x(a(x)\partial_x \varphi)- \partial_x(f'(U^\e) \varphi)= \lambda^\e \varphi, \quad \varphi(0)=\varphi( \ell)=0.
\end{equation*}
Firstly, we show that the eigenvalues of $\mathcal L^\e_\xi$ are real. To this aim, let us introduce the self-adjoint operator
\begin{equation*}
\mathcal N^\e_{\xi(t)} \psi := \e \partial_x(a(x)\partial_x \psi)- W^\e(x;\xi(t)) \psi,
\end{equation*}
where
\begin{equation}\label{defW}
W^\e(x;\xi(t)):= \frac{1}{a(x)}\left(\frac{f'(U^\e)}{2}\right)^2+\frac{1}{2} \, \e \, \partial_x f'(U^\e).
\end{equation}
A straightforward computation shows  that $\varphi^\e$ is an eigenfunction for $\mathcal L^\e_\xi$ relative to the eigenvalue $\lambda^\e$ if and only if
\begin{equation*}
\psi^\e(x;\xi)= \exp\left( -\frac{1}{2\e} \int_{x_0}^x \frac{f'(U^\e)(y;\xi)}{a(y)} dy \right) \varphi^\e(x;\xi)
\end{equation*}
is an eigenfunction for the operator $\mathcal N^\e_\xi$ relative to the eigenvalue $\mu^\e = \e \lambda^\e$. Hence
\begin{equation}\label{spettriuguali}
\e \, \sigma(\mathcal L^\e_\xi ) \equiv \sigma(\mathcal N^\e_\xi ),
\end{equation}
so that, since $\mathcal N^\e_\xi$ is self-adjoint, we can state the the spectrum of $\mathcal L^\e_\xi$ is composed by {\bf real eigenvalues}.

Going further, if $u$ is an eigenfunction of $\mathcal L^\e_\xi$ relative to the first eigenvalue $\lambda_1^\varepsilon$,
integrating in $(-\ell,\ell)$ the relation $\mathcal{L}^\varepsilon_\xi u=\lambda_1^\e\,u$, we deduce the identity  
\begin{equation*}
	0=\int_{-\ell}^{\ell}\bigl(\mathcal{L}^{\varepsilon}_\xi-\lambda_1^\varepsilon\bigr)u\,dx
		=\varepsilon \,(a(l)u'(l)-a(-l)u'(-l))-\lambda_1^\varepsilon\int_{-\ell}^{\ell} u(x)\,dx.
\end{equation*}
Assuming, without loss of generality, $u$ to be strictly positive in $(-\ell,\ell)$ and since $a(x)>0$  by assumption, we get $
	\lambda_1^\varepsilon < 0 $. Hence, there holds
\begin{equation*}
	\sigma(\mathcal{L}^{\varepsilon}_\xi)\subset (-\infty,0).
\end{equation*}

\begin{remark}
{\rm
With analogous computations, it is possible to prove that the eigenvalues of the linearized operator obtained from a linearization around the {\bf exact} steady state $\bar U^\e(x)$ are all negative;  this shows  the asymptotic stability of $\bar U^\e(x)$.
}
\end{remark}

\subsubsection{Estimates for the first eigenvalue}

We mean to control from below  $\lambda_1^\e$. To this aim, we estimate the first eigenvalue $\mu^\varepsilon_1$ of the operator $\mathcal{N}^\varepsilon_\xi$, and we use the relation \eqref{spettriuguali}; by means of the inequality
\begin{equation*}
		|\mu_1^\varepsilon|\leq \frac{|\mathcal N^\varepsilon_\xi\,\psi|_{{}_{L^2}}}{|\psi|_{{}_{L^2}}},
\end{equation*}
that holds for smooth test function $\psi$ such that $\psi(\pm \ell)=0$,  we look for a test function $\psi$ such that $\psi(x):=\psi_0^\varepsilon(x)-K^\varepsilon(x)$, where
\begin{equation*}
\psi_0^\varepsilon(x):=\exp\left(\frac{1}{2\varepsilon}\int_{\xi}^{x}\frac{f'(U^\varepsilon(\xi;y))}{a(y)}\,dy\right),
\end{equation*}
and such that there holds
\begin{equation*}
\mathcal{N}^\varepsilon_\xi \psi:=W^\varepsilon\,K^\e.
\end{equation*}
A direct computation, show that $K^\e$ has to solve
\begin{equation*}
\left\{\begin{aligned}
&\partial_x \left( a(x) \partial_x K^\e \right)=0, \\
& K^\e(\pm \ell)=\psi_0^\e(\pm \ell).
\end{aligned}\right.
\end{equation*}
Hence, by integrating, we get
\begin{equation*}
		K^\varepsilon(x):=\left\{\psi^\varepsilon_0(-\ell)+\psi^\varepsilon_0(\ell)\right \} \left( \int_{-\ell}^\ell \frac{dx}{a(x)} \right)^{-1} \, \int_{-\ell}^x \frac{dy}{a(y)}  + \psi_0^\e(-\ell).
\end{equation*} 
Going further,  there holds
\begin{equation*}
		|\mu^\varepsilon_1|\leq \frac{|W^\varepsilon\,K^\varepsilon|_{{}_{L^2}}}{|\psi^\varepsilon_0-K^\varepsilon|_{{}_{L^2}}}
			\leq C\,\frac{|K^\varepsilon|_{{}_{L^2}}}{|\psi^\varepsilon_0|_{{}_{L^2}}-|K^\varepsilon|_{{}_{L^2}}}
			= \frac{C}{|K^\varepsilon|_{{}_{L^2}}^{-1}|\psi^\varepsilon_0|_{{}_{L^2}}-1}
\end{equation*}
as soon as $|\psi^\varepsilon_0|_{{}_{L^2}}>|K^\varepsilon|_{{}_{L^2}}$. We assume $\psi_0(-\ell)\geq \psi_0(\ell)$, the opposite case being similar;
from the definition of $K^\varepsilon$ and from the properties of $a(x)$, it follows
\begin{equation*}
	|K^\varepsilon|_{{}_{L^2}}^2
		\leq 2\ell\,\psi_0^2(-\ell).
\end{equation*}
so that
\begin{equation*}
	|K^\varepsilon|_{{}_{L^2}}^{-2}|\psi^\varepsilon_0|_{{}_{L^2}}^2
		\geq \frac{1}{2\ell}\,\psi_0^{-2}(-\ell)\,\int_{-\ell}^{\ell}  |\psi^\varepsilon_0(x)|^2\,dx.
\end{equation*}
We define 
\begin{equation*}
	I^\varepsilon:=\int_{-\ell}^{\ell}  \exp\left(\frac{1}{\varepsilon}\int_{-\ell}^{x}\frac{f'(U^\varepsilon(\xi;y))}{a(y)}\,dy\right)\,dx.
\end{equation*}
Since $U^\varepsilon$ converges to the step function $U^0:=u_- \, \chi_{(-\ell,\xi)} + u_+ \, \chi_{(\xi,\ell)}$ as $\varepsilon\to 0^+$,
we get
\begin{equation*}
\begin{aligned}
	I^\varepsilon &=\int_{-\ell}^{\ell}  \exp\left(\frac{1}{\varepsilon}\int_{-\ell}^{x}\frac{1}{a(y)}(f'(U^\varepsilon)-f'(U^0))(\xi;y)\,dy\right)
			\exp\left(\frac{1}{\varepsilon}\int_{-\ell}^{x}\frac{f'(U^0(\xi;y))}{a(y)}\,dy\right)\,dx \\
	&\geq  e^{-|f'(U^\varepsilon)-f'(U^0)|_{{}_{L^1}}/\beta\,\varepsilon}\,I^0,
	\end{aligned}
\end{equation*}
where we used $a(x) \leq \beta$. Moreover there holds
\begin{equation*}
	\begin{aligned}
	I^0&= \int_{-\ell}^{\ell} {\rm exp}\left(\frac{1}{\e} \left[\int_{-\ell}^\xi \frac{f'(u^-)}{a(y)} \, dy + \int_{\xi}^{x}\frac{ f'(u^+)}{a(y)} \, dy \right] \right) \, dx \\
	&\geq \int_{-\ell}^{\ell} e^{ f'(u^-) (\xi+\ell)/\beta \e\ +  f'(u^+) (x-\xi)/\beta \e} \, dx \\
	&= e^{\frac{1}{\beta \,\e} f'(u^-) (\xi+\ell)} \, \int_{-\ell}^\ell e^{ \frac{1}{\beta \,\e} f'(u^+) (x-\xi)} \, dx \sim e^{\frac{1}{{\beta}\,\e} f'(u^-) (\xi+\ell)} \, \e.
 	\end{aligned}
\end{equation*}
In particular, if we suppose $|f'(U^\varepsilon)-f'(U^0)|_{{}_{L^1}}\leq c_0\varepsilon$ for some $c_0>0$, we end up with
\begin{equation*}
	|K^\varepsilon|_{{}_{L^2}}^{-1}|\psi^\varepsilon_0|_{{}_{L^2}}	\geq C_1\, \sqrt{\e} \, e^{C_2/\varepsilon}.
\end{equation*}
Thus, for the first eigenvalue $\mu^\varepsilon_1$ of the self-adjoint operator $\mathcal{N}^\varepsilon_\xi$ there holds
the estimate $|\mu^\varepsilon_1|\leq \left(C_1\, {\sqrt{\e}}\,e^{C_2/\varepsilon}\right)^{-1}$ for some positive constant $C_1, C_2$.
As a consequence,  since the spectrum $\sigma(\mathcal{L}^\varepsilon_\xi)$ coincides with 
$\varepsilon^{-1}\sigma(\mathcal{N}^\varepsilon_\xi)$, there holds
\begin{equation}\label{stimaprimoautovalore}
-C\,e^{-C/\varepsilon}\leq \lambda^\varepsilon_1<0.
\end{equation}
\begin{remark}
{\rm
Since $f'$ is a continuous function, the request $|f'(U^\varepsilon)-f'(U^0)|_{{}_{L^1}}\leq c_0\varepsilon$ is satisfied if we require $|U^\varepsilon-U^0|_{{}_{L^1}}\leq c_0\varepsilon$, which is consistent with the convergence of $U^\varepsilon$ to $U^0$ as $\e \to 0$.
}
\end{remark}

\subsubsection{Estimate from above for the second eigenvalue}

\noindent We mean to give an estimate on the behavior of the second and subsequent eigenvalues of the operator $\mathcal L^\e_\xi$. To this aim,  we need some additional assumptions on the limiting behavior of the functions ${a^\e(x)=}f'(U^\varepsilon(x))$ as $\varepsilon\to 0^+$.
Precisely, inspired by \cite{MS},
we suppose that $a^\varepsilon\in C^1([-\ell,\ell])$ and $a\in L^\infty([-\ell,\ell])$ satisfy the following hypotheses:\\\\
{\bf i.} $a^\varepsilon\in C^2([-\ell,\ell]\setminus\{\xi\})$, $a\in C^1([-\ell,\ell]\setminus\{\xi\})$ and
\begin{equation*}
\begin{cases}
\displaystyle{\frac{da^{\varepsilon}}{dx}, \frac{d^2a^{\varepsilon}}{dx^2}<0<a^{\varepsilon},\frac{da}{dx}\quad\textrm{in}\;(-\ell,\xi)}\\\\
\displaystyle{a^{\varepsilon},\frac{da^{\varepsilon}}{dx},\frac{da}{dx}<0<\frac{d^2a^{\varepsilon}}{dx^2}\quad\textrm{in}\;(\xi,\ell),}
\end{cases}
\end{equation*}
\vspace{0.2cm}
{\bf ii.} there exist the left/right first order derivatives of $a^\varepsilon$ at $\xi$ and 
\begin{equation*}
\varepsilon\left|\frac{da^\varepsilon}{dx}\right|\leq C \qquad\textrm{and}\qquad \liminf_{\varepsilon\to 0^+} \,\varepsilon\left|\frac{da^\varepsilon}{dx}(\xi\pm)\right|>0.
\end{equation*}
\vspace{0.2cm}
{\bf iii.} for any $C>0$ there exists $c_{{}_{0}}>0$ such that, if $|x-\xi|\geq c_{{}_{0}}\varepsilon$, then
\begin{equation*}
|f'(U^\varepsilon)-f'(U^0)|\leq C\,\varepsilon,
\end{equation*}
where we recall $f'(U^0)(x):= f'(u_-) \chi_{(-\ell,\xi)} + f'(u_+) \chi_{(\xi,\ell)}$.
\vskip0.2cm
Under these assumptions, it is possible to prove the following Lemma describing 
 the function $W^\varepsilon+\varepsilon\lambda^\varepsilon$, with $W^\varepsilon$ given in \eqref{defW}.

\begin{lemma}\label{lem:propW} 
Let the family $a^\varepsilon$ and the function $a$ be such that assumptions \eqref{ipoa}, {\sf A1-2-3} are satisfied, and let $\lambda^\varepsilon<0$ be such that
\begin{equation*}
\inf_{\varepsilon>0}\varepsilon\lambda^\varepsilon>-\frac{1}{4\beta}\,\alpha_{{}_{0}}^2
\qquad\textrm{where }\; \alpha_{{}_{0}}:=\min\{|f'(u^-)|,|f'(u^+)|\}.
\end{equation*}
Then there exist $\varepsilon_{{}_{0}}>0$ such that, for $\varepsilon<\varepsilon_{{}_{0}}$, the function $W^\varepsilon+\varepsilon\lambda^\varepsilon$ enjoys the following properties:\\\\
{\bf i.} $W^\varepsilon+\varepsilon\lambda^\varepsilon$ is decreasing in $(-\ell,\xi)$
and increasing in $(\xi,\ell)$;\\\\
{\bf ii.} there exist $C, c>0$ such that, for any $x$ with $|x-\xi|\geq c\,\varepsilon$ there holds $W^\varepsilon+\varepsilon\lambda^\varepsilon\geq C>0$;\\\\
{\bf iii.} there exist the left/right limits of $W^\varepsilon+\varepsilon\lambda^\varepsilon$ at $\xi$ and $\beta:=\limsup\limits_{\varepsilon\to 0^+} \bigl(W^\varepsilon(\xi\pm)+\varepsilon\lambda^\varepsilon\bigr)<0.$
\end{lemma}

\begin{proof}
The proof lies on a straightforward application of the properties {of the functions} $f'(U^\e)$ {and $a(x)$}, and closely resemble the one of \cite[Lemma $4.3$]{MS}.
\end{proof}

\begin{remark}{\rm 

In the easiest case of the Burgers equation, i.e. $a(x)\equiv1$ and $f(u)=u^2/2$, from the explicit expression of $U^\e$ given in \eqref{defUe} and since $f'(u)=u$, it is easy to check that the assumption we made on $f'(U^\e)$ are satisfied (see also Figure 2).

\begin{figure}
\centering
\includegraphics[width=14cm,height=7cm]{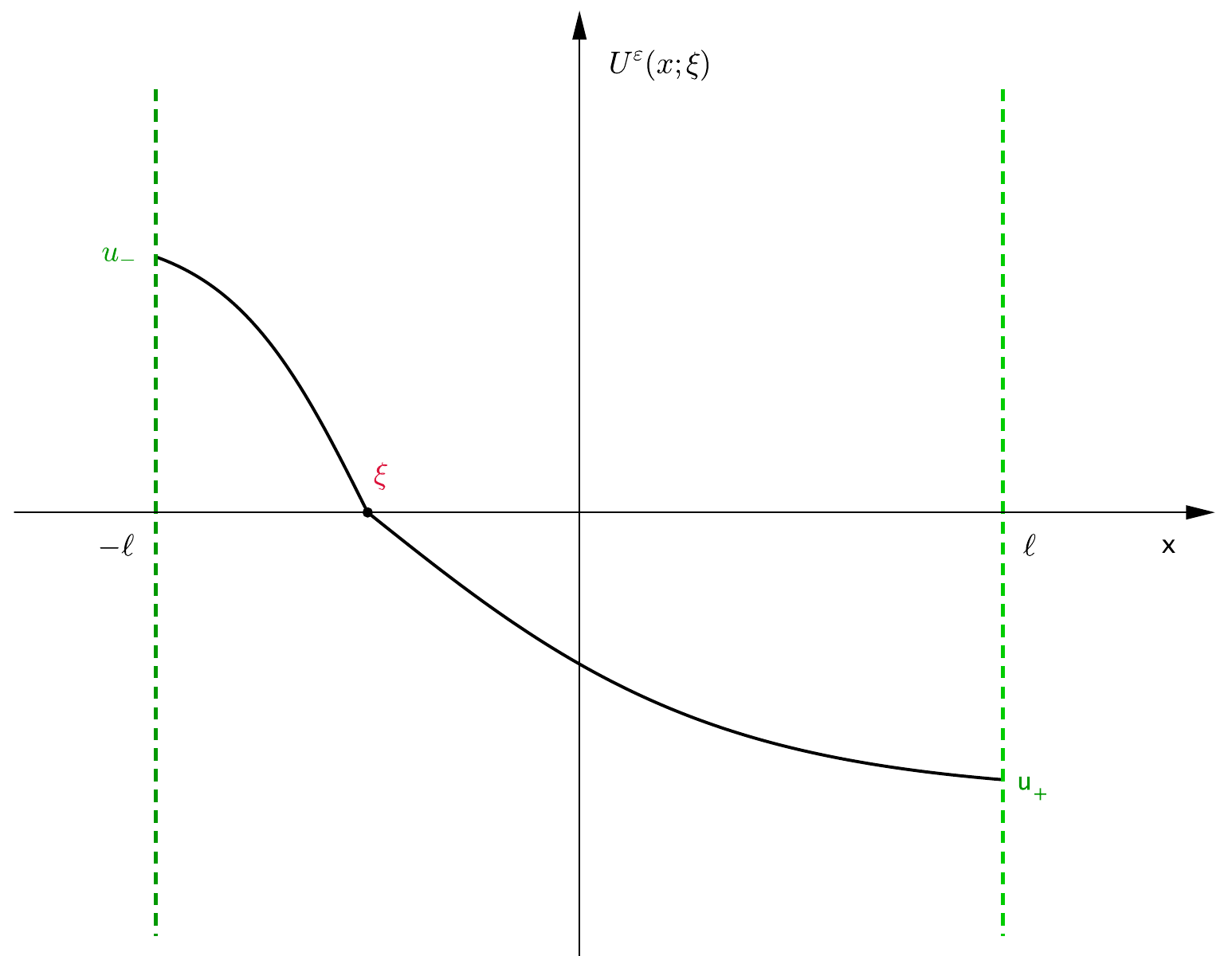}
\caption{\small{The approximate steady state for the Burgers equation. $U^\e$ is obtained by matching two exact steady states in the intervals $(-\ell,\xi)$ and $(\xi,\ell)$; as a consequence, $U^\e$ is a $C^0$ function but its first order derivative has a jump located in $x=\xi$. }}
\end{figure}

}
\end{remark}

From Lemma \ref{lem:propW} we can infer that, in the regime $\varepsilon\to0^+$, $W^\varepsilon+\varepsilon\lambda^\varepsilon$ has two zeros in $[-\ell,\ell]$, denoted here by $y^\varepsilon_\pm$; moreover 
$$-\ell<y^\varepsilon_-<\xi<y^\varepsilon_+<\ell, \quad {\rm and} \quad |y^\varepsilon_\pm-\xi|\leq c_{{}_{0}}\,\varepsilon. $$ 

\vskip.25cm

Let $\lambda_2^\varepsilon$ and $\mu_2^\varepsilon=\varepsilon\,\lambda_2^\varepsilon$ be the second eigenvalues of the operators ${\mathcal L}^\varepsilon_\xi$ and ${\mathcal M}^\varepsilon_\xi$ respectively, with corresponding eigenfunctions $\varphi_2^\varepsilon$ and $\psi_2^\varepsilon$ such that \begin{equation}\label{relpsiphi}
 \psi^\e(x;\xi)= \exp\left( -\frac{1}{2\e} \int_{x_0}^x \frac{f'(U^\e)(y;\xi)}{a(y)} dt \right) \varphi^\e(x;\xi),
 \end{equation}
 and let assume that Lemma \ref{lem:propW} holds with $\lambda^\e=\lambda_2^\varepsilon$. {We stress that we can choose $\lambda^\e=\lambda_2^\varepsilon$ without loss of generality since, if not possible, then it would follow $\lambda_2^\varepsilon\sim-1/\e^\alpha$ with $\alpha>1$, which implies  that \textbf{H2} is trivially satisfied.}

Since $\lambda_2^\varepsilon$ is the second eigenvalue, {applying the Sturm-Liouville theory to the operator $\mathcal N^\e_\xi$, we deduce} that the functions $\varphi_2^\varepsilon$ and 
$\psi_2^\varepsilon$ possess a single root located at some point $x_0^\varepsilon\in(-\ell,\ell)$. The sign properties of $W^\varepsilon+\mu_2^\varepsilon$ described in Lemma \ref{lem:propW} 
imply that $x_0^\varepsilon\in(y_-^\varepsilon, y_+^\varepsilon)$. Indeed, by contradiction, it is easy to verify that if $x_0^\varepsilon\in(-\ell,y^\varepsilon_-)\cup(y^\varepsilon_+,\ell)$, then it would  exists at least one point $\tilde{x}\in(-\ell,y^\varepsilon_-)\cup(y^\varepsilon_+,\ell)$ such that 
$$\partial_x\psi_2^\varepsilon(\tilde{x})=0, \qquad 
\partial_x^2\psi_2^\varepsilon(\tilde{x})\not=0, \qquad 
\psi_2^\varepsilon(\tilde{x})\not=0,
$$ 
and, using the equation solved by $\psi_2^\varepsilon$, this would imply that $\partial_x^2\psi_2^\varepsilon(\tilde{x})$ and $\psi_2^\varepsilon(\tilde{x})$ have the same sign, which is not possible. {Now} $\varphi_2^\varepsilon$ and $\psi_2^\varepsilon$ restricted to the intervals $(-\ell,x_0^\varepsilon)$ and $(x_0^\varepsilon,\ell)$ are eigenfunctions relative to the first eigenvalue of the same operator 
considered in the corresponding  intervals and  with Dirichlet boundary conditions.
 
 Without loss of generality, we can assume  $x^{\e}_0\geq \xi$ and $\varphi_2^\e\geq0$ in $(x_0^\e,\ell)$ and we can restrict our attention  to the interval $J=(x^{\e}_0,\ell)$. By proceeding as in \cite{MS}, integrating on $J$, we get
\begin{equation*}
\lambda^\e_2\,\int_{x^{\e}_0}^{\ell} \phi^{\e}_2\,dx=\varepsilon\,\bigl(a(l)\partial_{x}\phi^\e_2(\ell)-a(x_0^{\e})\partial_{x}\phi^\e_2(x_0^{\e})\bigr)
 <-\varepsilon\,a(x_0^{\e})\partial_{x}\phi^\e_2(x_0^{\e}).
\end{equation*}
If we now assume $\psi^\e_2$ to be as in \eqref{relpsiphi} with $x_0=x^\e_0$ and renormalized so that $\max\psi^\e_2=1$, from the previous equality we deduce
\begin{equation}\label{estsecond1}
|\lambda_2|>\varepsilon a(x_0^\e)\,\partial_{x}\psi^\e_2(x^\e_0)\,I_{\e}^{-1},
\end{equation}
where
\begin{equation*}
I_\e:=\int_{x^\e_0}^{\ell} \exp \left(\frac{1}{2\varepsilon} \int_{x^\e_0}^x\frac{f'(U^{\varepsilon}(y))}{a(y)}\,dy \right)\,dx.
\end{equation*}
In order to get an estimate from below on $|\lambda_2^\e|$, we give an estimate from above on $I_\varepsilon$ and an estimate from below on $\partial_{x}\psi^\e_2(x^\e_0)$.
\begin{equation*}
\begin{aligned}
I_\varepsilon &\leq e^{|a^\varepsilon-a^0|_{{}_{L^1}}/2\varepsilon\alpha}
\int_{x^\e_0}^{\ell} e^{f'(u^-)(x-x^\e_0)/2\varepsilon\alpha}\,dx
=\frac{2\varepsilon\alpha}{|f'(u^-)|}\,e^{|a^\varepsilon-a^0|_{{}_{L^1}}/2\varepsilon\alpha}
\bigl(1-e^{f'(u^-)(\ell-x^\e_0)/2\varepsilon\alpha}\bigr)\\
&\leq\frac{2\varepsilon\alpha}{|f'(u^-)|}\,e^{|a^\varepsilon-a^0|_{{}_{L^1}}/2\varepsilon\alpha}
\leq C\,\varepsilon
\end{aligned}
\end{equation*}
where the last inequality holds since $|a^\varepsilon-a^0|_{{}_{L^1}}\leq C\,\varepsilon$.
Hence,  \eqref{estsecond1} becomes 
\begin{equation}\label{estsecond2}
|\lambda^\e_2|>C\,\frac{d\psi^\e_2}{dx}(x^\e_0)
\end{equation}
for some $C>0$ independent on $\varepsilon$. Now, let $x_M\in[-\ell,\ell]$ be such that $\psi^\varepsilon_2(x_M)=1$; from the properties of $W^\varepsilon+\varepsilon\,\lambda$ stated in Lemma \ref{lem:propW}, it follows that $x_M\in(x^\e_0,y_+)$. Then, by Lagrange Theorem, there exists $x_L\in(x^\e_0,x_M)$ such that
\begin{equation*}
\frac{d\psi^\e_2}{dx}(x_L)=\frac{1}{x_M-x^\e_0}\geq \frac{1}{y_+-\xi}\geq \frac{1}{c_{{}_{0}}\varepsilon}.
\end{equation*}
Since $\psi_2^\e(x_0^\e)=\psi_2^\e(\ell)=0$ and $\psi_2^\e>0$ in $(x_0^\e,\ell)$, we can infer that the function $\psi_2^\e$ is concave in the interval $(x^\e_0,y_+)$, deducing that
\begin{equation*}
\frac{d\psi^\e_2}{dx}(x^\e_0)\geq \frac{d\psi^\e_2}{dx}(x_L) \geq \frac{1}{c_{{}_{0}}\varepsilon}.
\end{equation*}
In conclusion, from \eqref{estsecond2}, we deduce
\begin{equation}\label{stimasecondoautovalore}
|\lambda^\e_2|\geq \frac{C}{\varepsilon}\Longrightarrow\lambda^\e_2\leq -\frac{C}{\varepsilon}
\end{equation}
for some $C$ independent on $\varepsilon$. 

Estimates \eqref{stimaprimoautovalore} and \eqref{stimasecondoautovalore} show that hypotheses {\bf H2-H3-H4} are satisfied in the case of a quasilinear viscous  conservation law.

\subsection{The speed rate of convergence of the shock layer}

We here mean to obtain an asymptotic expression for the term $\theta^\e(\xi)$; indeed, recalling the equation for $\xi$ in \eqref{NL}, and supposing the perturbation $v$ to be small, we have
\begin{equation*}
\frac{d\xi}{dt} \approx \theta^\e(\xi).
\end{equation*}
Hence, the function $\theta^\e$ gives a good approximation of the speed  rate of convergence of the solution towards its asymptotic steady state.

Without loss of generality, we may assume $x^*$, the unique zero of the exact steady state, to be equal to zero.

\vskip0.2cm
Since $\theta^\e(\xi):= \langle \psi_1^\e(\xi), \mathcal P^\e[U^\e]\rangle$, we need an expression for $\psi_1^\e$, the first eigenfunction of the adjoint linearized operator $\mathcal L^{\e,*}_\xi$, defined as
\begin{equation*}
\mathcal L^{\e,*}_\xi v := \e\partial_x \left( a(x)\partial_x v\right)+f'(U^\e) \partial_x v.
\end{equation*}
Following the idea of \cite{MS}, for $\varepsilon\sim 0$, the eigenfunction $\psi_1^\varepsilon$ is close to the eigenfunction
of  ${\mathcal L}^{0,\ast}_\xi$ relative to the eigenvalue $\lambda=0$, where
\begin{equation*}
	f'(U^{0})(x;\xi):=f'(u_-)\chi_{{}_{(-\ell,\xi)}}(x)+f'(u_+)\chi_{{}_{(\xi,\ell)}}(x)
\end{equation*}
Hence, $\psi_1^0(x)$ solves
\begin{equation*}
\left\{\begin{aligned}
&\e\partial_x (a(x)\partial_x \psi_1^0)+f'(u_-) \partial_x \psi_1^0 =0, \quad {\rm for } \ \  x \in (-\ell,\xi), \\
&\e\partial_x (a(x)\partial_x \psi_1^0)+f'(u_+) \partial_x \psi_1^0 =0, \quad {\rm for } \ \  x \in (\xi, \ell), \\
&\psi_1^0(-\ell)=0, \quad \psi_1^0(\ell)=0, \quad \ldbrack \psi_1^0 \rdbrack_{x=\xi}=0.
\end{aligned}\right.
\end{equation*}
By integrating in $(-\ell, \xi)$ and $(\xi,\ell)$ respectively, and by imposing the boundary conditions and the condition on the jump, we obtain the following expression for $\psi_1^0$
\begin{equation*}
	\psi_1^0(x):=\left\{\begin{aligned}
			&(1-e^{f'(u_+)(b(\ell)-b(\xi))/\varepsilon})(1-e^{-f'(u_-)(b(\ell)+b(x))/\varepsilon})
					&\qquad	&x<\xi,\\
			&(1-e^{-f'(u_-)(b(\ell)+b(\xi))/\varepsilon})(1-e^{f'(u_+)(b(\ell)-b(x))/\varepsilon})
					&\qquad	&x>\xi,		
		\end{aligned}\right.
\end{equation*}
being $b(x):= \int a^{-1}(x) \, dx$. In particular, in the limit $\varepsilon\to 0$, we obtain
$\psi_1^\varepsilon\approx 1$ so that
\begin{equation*}
 	\theta^\varepsilon(\xi)\approx \langle 1,{\mathcal P^\e[U^{\varepsilon}] \rangle} \approx
		e^{-C/\varepsilon}.
\end{equation*}
This estimate show that the speed of the interface is  {\it exponentially small} when $\e$ is small; for example, in the special case of the Burgers equation there holds
\begin{equation*}
\theta^\e (\xi) \sim e^{-u_*(\ell+\xi)/\varepsilon}-e^{-u_*(\ell-\xi)/\varepsilon},
\end{equation*}
showing that the hypotheses stated in Proposition \ref{SMshock}, equation \eqref{ipoteta}, are satisfied.

\section{Appendix A}

In this Appendix we collect some useful results obtained in \cite{Pazy83}.
Let us consider the initial value problem 
\begin{equation}\label{pazy}
\partial_t u = A(t) u+ f(t), \quad u(s)=u_0 \qquad 0 \leq s \leq t \leq T.
\end{equation}

\begin{ans}\label{def1}
Let $X$ a Banach space. A family $\{ A(t)\}_{t \in [0,T]}$ of infinitesimal generators of $C_0$ semigroups on $X$ is called stable if there are constants $M \geq 1$ and $\omega$ (called the stability constants) such that
\begin{equation*}
(\omega,+\infty) \subset \rho(A(t)), \quad {\rm for} \quad t \in [0,T]
\end{equation*}
and
\begin{equation*}
\left \| \Pi_{j=1}^k R(\lambda: A(t_j))\right \|  \leq M(\lambda-\omega)^{-k}, \quad {\rm} 
\end{equation*}
for $\lambda>\omega$ and for every finite sequence $0 \leq t_1 \leq t_2, . . . . , t_k \leq T$, $k=1,2,....$.
\end{ans}
If, for $t \in [0,T]$, $A(t)$ is the infinitesimal generator of a $C_0$ semigroup $S_t(s)$, $s \geq 0$ satisfying $\| S_t(s)\| \leq e^{\omega s}$, then the family $\{ A(t)\}_{t \in [0,T]}$ is clearly stable with constants $M=1$ and $\omega$. Precisely, if  the operator $A(t)$ generates a $C_0$ semigroup $S_t(s)$ for every fixed $t \in [0,T]$, and we can find an estimate for $\| S_t(s)\|$ that is independent of $t$, then the whole family $\{ A(t)\}_{t \in [0,T]}$ is stable in the sense of Definition \ref{def1}.

\begin{theorem}\label{thpaz1}
Let $\{ A(t)\}_{t \in [0,T]}$ be a stable family of infinitesimal generators with stability constants $M$ and $\omega$. Let $B(t)$, $0 \leq t \leq T$ be a bounded linear operators on $X$. If $\| B(t)\| \leq K$ for all $t \leq T$, then $\{ A(t)+ B(t)\}_{t \in [0,T]}$ is a stable family of infinitesimal generators with stability constants $M$ and $\omega+ MK$.
\end{theorem}

Now we prove the existence of the so  called {\it evolution system} $U(t,s)$ for the initial value problem \eqref{pazy}, that is a generalization of the semigroup generated by a linear operator $A$, when such operator depends on time.  To this aim, let us state the following result (for more details, see \cite[Theorem 2.3, Theorem 3.1, Theorem 4.2]{Pazy83}).

\begin{theorem}\label{thpaz3}
Let $\{ A(t)\}_{t \in [0,T]}$ be a stable family of infinitesimal generators of $C_0$ semigroups on $X$. If $D(A(t))=D$, {that is the domain of $A(t)$ is independent on $t$,} and for $u_0 \in D$, $A(t)u_0$ is continuously differentiable in $X$, then there exists a unique evolution system $U(t,s)$, $0 \leq s \leq t \leq T$, satisfying 
 \begin{equation*}
 \| U(t,s) \| \leq Me^{\omega (t-s)}, \quad {\rm for } \quad 0 \leq s \leq t \leq T.
 \end{equation*}
Morevoer, if $f \in C([s,T],X)$, then, for every $u_0 \in X$, the initial value problem \eqref{pazy} has a unique solution given by 
 \begin{equation*}
 u(t)= U(t,s)u_0 + \int_s^t U(t,r) f(r) \ dr.
 \end{equation*}
for all $0 \leq s \leq t \leq T$.
\end{theorem}


\end{document}